\documentclass[11pt]{article}

\usepackage{lipsum}
\usepackage{amsmath,amssymb,amsfonts}
\usepackage{hyperref}
\usepackage{graphicx}
\usepackage{epstopdf}
\usepackage{algorithmic}
\usepackage{caption}
\usepackage{subcaption}
\usepackage{color}

\usepackage[left=3.5cm,right=3cm,top=3cm,bottom=2.5cm]{geometry}
\usepackage{ulem,amsmath,amssymb,amsthm,mathrsfs,amsfonts,dsfont} 
\usepackage{epsfig,setspace}
\usepackage{graphicx}
\usepackage{url}
\usepackage{color}
\usepackage{ifthen}
\usepackage{epstopdf}
\usepackage{bbm}
\usepackage{caption}
\usepackage{inputenc}

\newtheorem{lemma}{Lemma}

\newtheorem{theorem}{Theorem}

\newboolean{showcomments}
\setboolean{showcomments}{true}

\graphicspath{ {./images/} }

\newcommand\NEW[1]{#1}
\newcommand{\remove}[1]{}

\newcommand{\N}{\widetilde{N}}
\newcommand{\ep}{\epsilon}
\newcommand{\ka}{\hat{k}}
\newcommand{\ha}{\hat{h}}
\newcommand{\V}{\widehat{V}}
\newcommand{\ca}{\hat{ca}}
\newcommand{\C}{\widetilde{C}}

\begin{document}

\title{Revisiting SIR in the age of COVID-19: Explicit Solutions and
  Control Problems}

\date{\today}

\author{Vivek S. Borkar and D. Manjunath\footnote{The authors are also
    affiliated with the Centre for Machine Intelligence and Data
    Science at IIT Bombay. Work supported in part by a grant
    `\textit{Controlling epidemics}' from the Science and Engineering
    Research Board, Government of India. VSB was also supported by the
    J.\ C.\ Bose and S.\ S.\ Bhatnagar Fellowships from the Government
    of India.}  \\
  Department of Electrical Engineering IIT  Bombay,\\
  Mumbai INDIA 400076. \\
  \texttt{borkar,dmanju@ee.iitb.ac.in}
  }

\maketitle

\begin{abstract}
  The non-population conserving SIR (SIR-NC) model to describe the
  spread of infections in a community is proposed and studied. Unlike
  the standard SIR model, SIR-NC does not assume population
  conservation.  Although similar in form to the standard SIR, SIR-NC
  admits a closed form solution while allowing us to model mortality,
  and also provides different, and arguably a more realistic,
  interpretation of the model parameters. Numerical comparisons of
  this SIR-NC model with the standard, population conserving, SIR
  model are provided. Extensions to include imported infections,
  interacting communities, and models that include births and deaths
  are presented and analyzed. Several numerical examples are also
  presented to illustrate these models.

  Two control problems for the SIR-NC epidemic model are
  presented. First we consider the continuous time model predictive
  control in which the cost function variables correspond to the
  levels of lockdown, the level of testing and quarantine, and the
  number of infections. We also include a switching cost for moving
  between lockdown levels. A discrete time version that is more
  amenable to computation is then presented along with numerical
  illustrations. We then consider a multi-objective and
  multi-community control where we can define multiple cost functions
  on the different communities and obtain the minimum cost control to
  keep the value function corresponding to these control objectives
  below a prescribed threshold.
\end{abstract}


  

\section{Introduction}
\label{sec:intro}

\subsection{Background}
\label{sec:background}
Epidemic models have been used in a variety of situations---modeling
of the spread of an infection in a community, spread of rumors in a
community, spread of viruses in systems of computer networks, etc. An
important workhorse in this domain has been the SIR model that was
introduced by Kermack and McKendrick \cite{Kermack27} and has received
wide textbook treatment, e.g., \cite{Daley99}. This classical model
assumes a community where everyone interacts with everyone in a
homogeneous manner and the average behavior of these interactions are
known. In contrast, in some new applications, the community
interactions are represented by a graph and the interactions happen
along the edges of the graph. See for example \cite{Ganesh05} for
models of SIR epidemics on graphs. \NEW{There are also stochastic
  models that explictly account for randomness, see, e.g.,
  \cite{Ball83}, \cite{Ball93}.}

The standard SIR model for a single community is a compartmental model
in which the population is divided into sub-populations of susceptible
(S), infected (I), and recovered (R) individuals. At time $t,$ let
$S(t)$ be the size of the population that is susceptible to be
infected by the epidemic, $I(t)$ be the number in the population that
is currently infected and carrying the infection, and $R(t)$ be the
number that have recovered from the infection and are neither going to
infect others nor will be infected again. Let $\mu$ be the contact
rate (contacts per unit time for each person) in the population and $p$ be the
fraction of meetings between an infected and susceptible person in
which the susceptible person gets infected. Let $\gamma$ be the rate
at which an infected person recovers from the infection. Let $N$ be
size of the population. Letting $\lambda= \mu p,$ the o.d.e.s
governing the evolution of the system are:
\begin{eqnarray}
  \dot{S}(t) & = & - \lambda \frac{I(t) S(t)}{N}, \nonumber\\
  \dot{I}(t) & = & \lambda \frac{I(t) S(t)}{N} - \gamma I(t),
  \nonumber\\
  \dot{R}(t) & = & \gamma I(t), \nonumber\\
  N & = & S(t) + I(t) + R(t) \equiv \ \mbox{a constant}.
  \label{eq:SIR}
\end{eqnarray} 
Parameter $\lambda$ is also called the pairwise infection
rate. Observe that in this model, deaths, hospitalizations, and
isolations are not modeled. In fact, it is assumed that those that
are infected will continue to spread the infection during the entire
time that they are carrying the infection, i.e., till they move into
the recovered state. 

Since the introduction of the SIR model, a myriad variations that
include many more compartments and more detailed interactions have
been proposed and studied. A key bottleneck in the study of these
models is the nonlinearity that defies easy solutions. Till recently,
explicit solutions of the basic system defined by \eqref{eq:SIR} were
not available. For example, \cite{Daley99} only provides a rather
complex approximation for $R(t).$ Solutions for some special cases are
available, see, e.g., \cite{Shabbir10}. More recently, a parametric
solution for the general case has been described in
\cite{Harko14}. Even here, converting the parametric form to a time
domain solution would require numerically solving an integral
equation. Furthermore, small changes to the model that would make the
model more realistic makes the search for a solution even more
hopeless, as was exasparatingly noted in \cite{Nucci04}\footnote{While
  explaining that removing the population conservation constraint
  would make solutions for the even simpler SIS model impossible, the
  authors remark ``It would seem that a fatal disease which this
  models is also not good for mathematics''. }.

In the absence of explicit solutions, the focus has been on numerical
evaluation and also on some key quantities of interest. For example,
in the SIR model of \eqref{eq:SIR}, it is easy to see that $\lambda
S(0) > \gamma$ will lead to a breakout of the epidemic. It is also
known that at the end of the epidemic, the number that will not be
infected will be strictly greater than 0, i.e., $S(t) \to c > 0.$
Several other key parameters are also of interest, especially if the
infections result in the need for hospitalizations and other
healthcare resources. The popular phrase in the context of the
COVID-19 pandemic has been ``flattening the curve'', i.e., to delay
and/or reduce the peak $I(t).$ Using the parametric solution of
\cite{Harko14}, we can get a closed form expression for the maximum
value of $I(t)$ (see \eqref{eq:SIR-Imax} in a later section) in the
SIR system of \eqref{eq:SIR}.  However, the time at which this peak is
achieved has to be obtained through a numerical solution to an
integral equation.

Furthermore, in the absence of explicit solutions of the system of
equations or of the key performance indicators, control problems are
even harder. Thus, not surprisingly, application of interventions by
suitably modulating either of $\lambda$ or $\gamma$ to achieve
specific control objectives is not well studied.\footnote{We remark
  here that this kind of control is different from vaccination based
  control whose effect is to reduce the susceptible population. This
  has received significant attention in the research literature.} In
this paper we first make an important and, what we believe to be, a
more realistic change to the basic SIR model in which we drop the
population conservation assumption of \eqref{eq:SIR} and propose the
SIR-NC (for `Non-Conservative') epidemic model. This is motivated by
first recognizing that those that recover neither infect the others
nor can they be infected by virtue of having developed immunity.  (The
SIR model confers immunity on those that have recovered. As of the
writing of this paper, reinfections in COVID-19 are rare and the SIR
is a very good model. The SIS model and its variants may be used for
cases where re-infections are more common.) Thus we can treat $R(t)$
as the part of the population that is `dropping out' or removed from
the system. The outflow from the infected population could also
include those that have died due to the infection, those that are
cured and are free from the infection, and also those that are
identified to be infected and hospitalized and/or quarantined. The
parameter $\gamma$ would then be the sum rate of all these. If one
wants to focus on, say, the recovered, denoted by $R(t)$, one would
then have $\dot{R}(t) = \beta I(t)$ where $\beta < \gamma$ is the
recovery rate. Similar formalism will be posited for the dead, the
quarantined, etc. Thus our SIR-NC model allows us to model
non-pharmacological interventions (NPIs) such as early isolation and
quarantining of those that are infected through testing. Importantly,
this is achieved without the addition of any new compartments to the
model.

The rest of the paper is organized as follows. In the next subsection
we introduce the SIR-NC model and present its solution in the
following section. In Section~\ref{sec:SIR-NC-model} we develop the
model and present Theorem~\ref{thm:SIR-NC} that provides an explicit
expression for $(I(t),S(t)),$ the peak value of $I(t),$ and the time
at which the peak is realized. The case of time-varying $\lambda$ and
$\gamma$ are also included. Theorem~\ref{thm:SIR-NC} is proved in
Section~\ref{sec:Thm1-proof}. Several variations of the basic model
are discussed in Section~\ref{sec:variations}---analysis of a
community with imported infections (Section~\ref{sec:SIR-NC-Import}),
interacting communities (Section~\ref{sec:SIR-NC-communities}), and
the SIR-NC model that allows for natural births and deaths
(Section~\ref{sec:birth-deaths}). Numerical examples will be used to
illustrate these models.  In Section~\ref{sec:SIR-NC-Control} we
discuss several control problems over a finite horizon for the SIR-NC
model after discussing the control objectives, specifically, minimum
cost control problems that control $I(t)$ so as to achieve specific
objectives. Sections~\ref{sec:cont-control} and
\ref{sec:discrete-control} develop the finite horizon continuous time
and discrete time control problems. In
Section~\ref{sec:multi-objective}, we develop a more general
multi-objective, multi-community control problem.  Illustrative
numerical results are also presented.

\subsection{The non-population conserving SIR (SIR-NC) model}
\label{sec:SIR-NC-intro}
We begin by recalling the philosophy behind the classical SIR
model. This model assumes that the population is comprised of three
sub-populations of susceptible, infected, and recovered people.  Let
$N(t) = S(t) + I(t) + R(t)$ denote the total population at time $t.$
There are no births or deaths, based on the assumption that these are
on a much slower time scale and can therefore be ignored. When a
susceptible person meets another person, nothing happens unless the
latter is infected; the fraction of such meetings with an infected
individual is $\frac{I(t)}{N(t)}.$ The infection happens with a rate
of $\lambda > 0$, whence the rate of transfer from susceptible to
infected is $\lambda \frac{S(t)I(t)}{N(t)}.$ The infected get
recovered at rate $\gamma > 0.$ The combined dynamics thus is
\begin{eqnarray*}
  \dot{S}(t) &=& -\lambda\frac{S(t)I(t)}{N(t)}, \\
  \dot{I}(t) &=& \lambda\frac{S(t)I(t)}{N(t)} - \gamma I(t), \\
  \dot{R}(t) &=& \gamma I(t).
\end{eqnarray*}
Summing the three yields $\dot{N}(t) = 0$, whence $N(t) \equiv$ a
constant $N$ (say). The above then simplifies to (\ref{eq:SIR}).
While this classical model has been an extensively studied and
successful model in epidemiology (\cite{Hethcote00} has an excellent
review, \cite{Daley99,Keeling08} are examples of excellent textbook
treatments) some of its limitations stand out:
\begin{enumerate}
\item While birth rate can be viewed as negligible, the death rate
  need not be so for lethal diseases such as Ebola or Covid-19.

\item Once you allow a nontrivial death rate, you do not have a
  constant population any more and the ensuing simplification
  disappears. Attempts to account for deaths while keeping the total
  population constant fly in the face of the microscopic picture we
  started with---it requires meeting the dead on equal footing with
  the living. 

\item Even in its pristine version, the dynamics is not easy to
  analyze and is usually studied only numerically or, as we have said
  previously, using messy approximations and hard to use parametric
  forms. 
\end{enumerate}

\NEW{Motivated by this, we present here an alternative model that goes
  back to \cite{Gleissner88}} and offers two clear advantages:
\begin{enumerate}
\item It is analytically tractable.
\item It is easy to incorporate nontrivial death rate into this model.
\end{enumerate}
Our model does not match the microscopic model which led to SIR as the
corresponding macroscopic behavior.  Nevertheless, we can interpret it
by introducing different time scales as we shall soon see in Lemma
\ref{twoscale} below. The key observation we start with is what is
already implicit in the SIR model above, viz., that the recovered
neither infect nor get infected (i.e., they become immune). They
affect the evolution of $S(t)$ and $I(t)$ only through the
probabilities of pairwise meetings between susceptible and infected,
because $R(t)$ enters the normalizing factor $N(t) \equiv N$. In our
model we count only those pairwise meetings that matter, i.e., the
meetings wherein a susceptible agent meets another susceptible or
infected person.  This is legitimate because the constraint that the
total population is conserved is now dropped, whence the recovered can
be legitimately viewed as having `dropped out'.  We add the additional
proviso that a susceptible person meets a large number of susceptible
or infected persons---not necessarily the whole population, but a
large enough statistically representative portion thereof. Then the
probability of a susceptible person meeting an infected person is
$\frac{I(t)}{\N(t)}$ where $\N(t) := S(t) + I(t)$. By a reasoning
analogous to the above, the equations now become (with an additional
tweak we mention later)
\begin{eqnarray*}
  \dot{S}(t) &=& -\lambda\frac{S(t)I(t)}{\N(t)}, \\
  \dot{I}(t) &=& \lambda\frac{S(t)I(t)}{\N(t)} - \gamma I(t), \\
  \dot{R}(t) &=& \beta I(t).
\end{eqnarray*}
The additional change we have made is to replace $\gamma$ by $0 <
\beta < \gamma$ in the third equation. That is, we allow for deaths
due to the disease (more generally, including quarantine etc.) at rate
$\gamma - \beta$. Clearly, $\N(t)$ cannot be expected to be a
constant, because the right hand side does not add to zero. This is
not a complication, but a blessing in disguise, because now $R(t)$
does not affect the evolution of $S(t)$ and $I(t)$ at all, in
conformity with our view that the recovered `fall out of the
system'. Furthermore, the interpretation of $\lambda$ in this model is
different. This is described below.
  
Even before we discuss this any further, we wish to underscore the
fact that this is \textit{not} a case of sub-sampling random pairwise
meetings that happen only one at a time on the microscopic scale in
the standard SIR model. That would lead to a different dynamics. In
other words, suppose that the individual meetings happen on a common
Poisson clock with a fixed rate and we choose to observe only those
between the susceptible and the infected or the susceptible. Any
reasonable continuum limit of this should also re-scale $\gamma$ as
$\gamma / \N(t)$ because of the time scaling implicit in the above
sampling. Our model is distinct from this scenario in that it requires
each person to meet a statistically significant portion of the entire
population `at one go'. This is, for example, what would happen at the
workplace, in the school, during the commute, etc. This is also the
view in the agent based simulation model implemented in
\cite{IISc_TIFR}. What this stands for is that the `atomic' or
individual pairwise meetings take place on a faster time scale than
the death rate due to the epidemic, which is a reasonable
assumption. A simple scenario would be the discrete dynamics
\begin{eqnarray*}
  S_{n+1} &=& S_n - \ep S_n\zeta_n, \\
  I_{n+1} &=& I_n + \ep I_n(\zeta_n - \gamma),
\end{eqnarray*}
where $\ep$ is the discrete time step, $\zeta_n$ is an indicator
random variable corresponding to a pairwise meeting which is $0$ if it
is benign (i.e., susceptible meeting susceptible) and $1$ if not
(i.e., susceptible meeting infected). The conditional probability that
$\zeta_n = 1$ given the past is proportional to the fraction of
infected population, say $\lambda\frac{I_n}{S_n + I_n}$. Thinking of
$\ep$ as a small time step, $\gamma$ remains removal (or recovery)
rate per unit time as $\ep\downarrow 0$, but the random meetings
$\{\zeta_n\}$ start crowding on the scale
$\Theta\left(\frac{1}{\ep}\right)$ per unit time. Then the averaging
effect kicks in and we get the desired dynamics in the limit. To make
this precise, define $S^\ep(t), I^\ep(t), t \geq 0,$ by
\begin{eqnarray*}
  S^\ep(n\ep) = S_n, \ I^\ep(n\ep) = I_n,
\end{eqnarray*}
with linear interpolation on $[n\ep, (n+1)\ep]$ for $n \geq 0$. Then
we have:

\begin{lemma}
  \label{twoscale}
  For any $T > 0$, $E\left[\left\|\sup_{t\in [0,T]}\|S^\ep(t) -
    S(t)\right\|^2\right] = O(\ep)$.
\end{lemma}

This is a special case of Lemma 1, p.\ 103, of \cite{Borkar08}.  In
fact, this is the standard averaging effect of two time scale
dynamics. Note that the interpretation of parameter $\lambda$ gets
modified from that for the original SIR model.

Getting back to our model, the most significant observation is what we
have already made, viz., that $S(t), I(t)$ obey a self-contained two
dimensional differential equation which, as it turns out, is
explicitly solvable as we shall see later.


\section{Analysis of the SIR-NC model}
\label{sec:SIR-NC-analysis}

\subsection{The model and its implications}
\label{sec:SIR-NC-model}

In this section we analyze the SIR-NC model in detail.  As we
elaborated in the previous section, the $R(t)$ only have a one-way
coupling with $I(t),$ the infected population.  We thus have that

\begin{equation}
  R(t) = \beta \int_0^t I(s)ds
  \label{eq:SIR-NC-R-of-t}
\end{equation}
which can be explicitly computed once $I(\cdot)$ is known.  Therefore
the SIR-NC system that we need to analyze is 
\begin{eqnarray}
  \dot{S}(t) &=& -\lambda\frac{I(t)S(t)}{S(t) + I(t)}, \nonumber \\
  \dot{I}(t) &=& \lambda\frac{I(t)S(t)}{S(t) + I(t)} - \gamma I(t).
  \label{eq:SIR-NC}
\end{eqnarray}
Of course, we should have $I(0) > 0$ and define $S(0)+I(0) = N(0).$ We
could also have $R(0) > 0$, but since they do not affect the evolution,
$R(0)=0$ is without loss of generality. This change to
non-conservation allows us to provide explicit expressions for $S(t),$
$I(t),$ $I_{\max} :=$ the peak value of $I(t),$ and $T_{\max} :=$ the
time at which it is attained (Theorem~\ref{thm:SIR-NC}). This in turn
helps us define several control problems with the many different
control objectives to shape the evolution of the epidemic.

It is easy to see that the condition for an epidemic to break out in
SIR-NC is
\begin{displaymath}
  \frac{\lambda}{\gamma} > 1+ \frac{I(0)}{S(0)}.
\end{displaymath}

The main result in this section is the following theorem. \NEW{(See
  also \cite{Gleissner88}.)}.

\begin{theorem}
  \label{thm:SIR-NC}
  Writing $C=\frac{S(0)}{I(0)},$ the solution to the SIR-NC system is
  as follows.
  \begin{align}
    I(t) & \ = \ I(0)e^{\int_0^t\left(\frac{\lambda Ce^{-(\lambda -
          \gamma)s}}{1 + Ce^{-(\lambda - \gamma)s}} - \gamma\right)ds}
    \nonumber \\
    & \ = \ I(0) \left(\frac{1+C e^{-(\lambda - \gamma)t}}{1+C}
    \right)^{-\left(\frac{\lambda}{\lambda - \gamma}\right)}
    e^{-\gamma t} \nonumber \\
    S(t) & \ = \ S(0)e^{-\int_0^t\left(\frac{\lambda }{1 +
        Ce^{-(\lambda - \gamma)s}}\right)ds} \nonumber \\    
    & \ = \ S(0) \left( \frac{e^{(\lambda - \gamma)t} + C}{1 +
      C}\right)^{-\left(\frac{\lambda}{\lambda - \gamma}\right)}
    \nonumber \\    
      T_{\max} & = \frac{\left(\log\left[C\left(\frac{(\lambda -
            \gamma)}{\gamma}\right)\right]\right)^+}{ \lambda -
        \gamma}.
    \label{eq:SIR-NC-results1}
  \end{align}
  
For the more general case of time-dependent $\lambda(t), \gamma(t), t
\geq 0$, we similarly have
  \begin{align}
    I(t) & \ = \ I(0)e^{\int_0^t\left(\frac{\lambda(s)
        Ce^{-\int_0^s(\lambda(r) - \gamma(r))dr}}{1 +
        Ce^{-\int_0^s(\lambda(r) - \gamma(r))dr}} -
      \gamma(s)\right)ds} \nonumber \\    
   S(t) & \ = \ S(0)e^{-\int_0^t\left(\frac{\lambda(s) }{1 +
       Ce^{\int_0^s(\lambda(r) - \gamma(r))dr}}\right)ds} \nonumber \\  
   T_{\max} & = \ \mbox{the solution of the equation} \nonumber\\
   \int_0^{T_{\max}}(\lambda(t) - \gamma(t))dt & \ =
   \ \left(\log\left[C\left(\frac{(\lambda(T_{\max}) -
       \gamma(T_{\max}))}{\gamma(T_{\max})}\right)\right]\right)^+.
    \label{eq:SIR-NC-results2}
  \end{align}
\end{theorem}

\noindent
\textbf{Remark} The last equation is arrived at simply by putting the
time derivative of $I(\cdot)$ equal to zero. This need not have a
unique solution, nor need every solution correspond to a maximum, even
a local maximum. 

We now compare the computations from this model with that of the
standard population conserving SIR model. First, $R(t)$ for the SIR-NC
is obtained by using $I(t)$ from \eqref{eq:SIR-NC-results1} in
\eqref{eq:SIR-NC-R-of-t}. Compare this with with the following
approximation from \cite{Daley99}. Here $\rho=\lambda/\gamma.$
%
\begin{align}
  \alpha & = \sqrt{ \frac{2 S(0)}{\rho^2} \left( N - S(0) \right) +
    \left( \frac{S(0)}{\rho} - 1 \right)^2 } \nonumber \\  
  \phi & = \tanh^{-1}\left( \frac{1}{\alpha} \left( \frac{S(0)}{\rho}
  - 1 \right) \right)  \nonumber \\  
  R(t) & \approx \ \frac{\rho^2}{S(0)} \left( \frac{S(0)}{\rho} -1
  \right) \ + \ \frac{\alpha \rho^2}{S(0)} \tanh\left(0.5\gamma \alpha
  t - \phi \right) \nonumber 
\end{align}
Next, compare the computation of $T_{\max}$ in SIR-NC as above with
that in the SIR model using the the parametric form of $S(t)$ and
$I(t)$ from \cite{Harko14}.  The following is adapted from
\cite{Toda20}.
\begin{align}
  S(t) & = S(0) u \nonumber  \\
  I(t) & = N \frac{\gamma}{\lambda} \log u + S(0)(1- u) + I(0)
  \nonumber \\
  t & = N \int_{u}^1 \frac{d\xi}{\xi(\lambda S(0) (1-\xi) + S(0) \lambda
    + N \gamma \log \xi)}
  \label{eq:SIR-Toda}
\end{align}
As can be seen, $t$ does not have a closed form solution.

Finally, we consider the computation of $I_{\max}.$ In the SIR model
we can see that the maximum of $I(t)$ occurs for $u = S(0)
\frac{\lambda}{\gamma}$ in (\ref{eq:SIR-Toda}) and the $I_{\max}$ in
the SIR system can be evaluated as
\begin{equation}
  I(0) + S(0) + N \frac{\gamma}{\lambda}\left( \log \left(\frac{N
    \gamma}{\lambda S(0)} \right) -1 \right) 
  \approx N (1 - \rho + \rho \log \rho) 
  \label{eq:SIR-Imax}
\end{equation}
where $\rho = \gamma/\lambda.$ In the approximation above, we have
assumed $I(0)$ is very small compared to $S(0)$ and hence $S(0)
\approx N.$
The corresponding expression for the SIR-NC system is
\begin{align}
  I_{\max}=& I(0) \left( \frac{\lambda}{(1+C)(\lambda-\gamma)}
  \right)^{-\frac{\lambda}{\lambda-\gamma}} \left(C
  \frac{\lambda-\gamma}{\gamma} \right)^{-\frac{\gamma}{\lambda-\gamma}}
  \label{eq:SIR-NC-Imax} \\
  = \ & S(0)^{-\frac{\gamma}{\lambda-\gamma}} \left( S(0) + I(0)
  \right)^{\frac{\lambda}{\lambda-\gamma}} (\lambda - \gamma) \left(
  \lambda^{-\lambda} \gamma^{\gamma} \right)^{\frac{1}{\lambda-\gamma}}
  \nonumber \\
  \approx & N(0) (\lambda - \gamma) \left( \lambda^{-\lambda}
  \gamma^{\gamma} \right)^{\frac{1}{\lambda-\gamma}} \nonumber \\
  = \ & N(0) (1-\rho)\rho^{\frac{\rho}{1-\rho}} 
    \label{eq:SIR-NC-Imax-approx}
\end{align}
where, as before, $\rho = \gamma/\lambda$ and the approximation
assumes that $I(0)$ is very small compared to $S(0).$

\begin{figure}[tbp]
  \begin{center}
  \begin{subfigure}{0.45\textwidth}
    \centering
    \includegraphics[width=\textwidth]{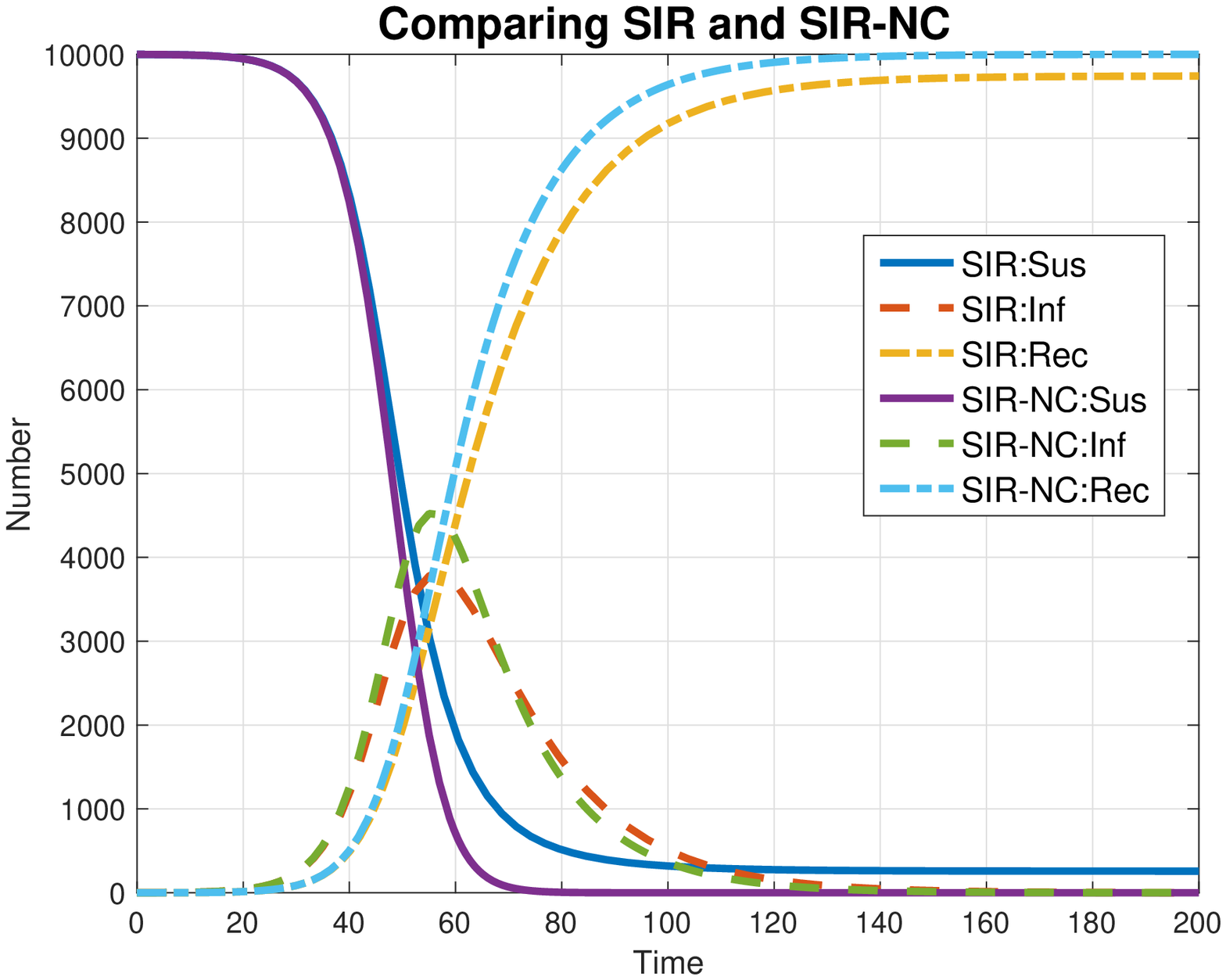}
    \caption{$N=10,000$}
    \label{plot:sir-sirnc-10k}
  \end{subfigure}
  \begin{subfigure}{0.45\textwidth}
    \centering
    \includegraphics[width=\textwidth]{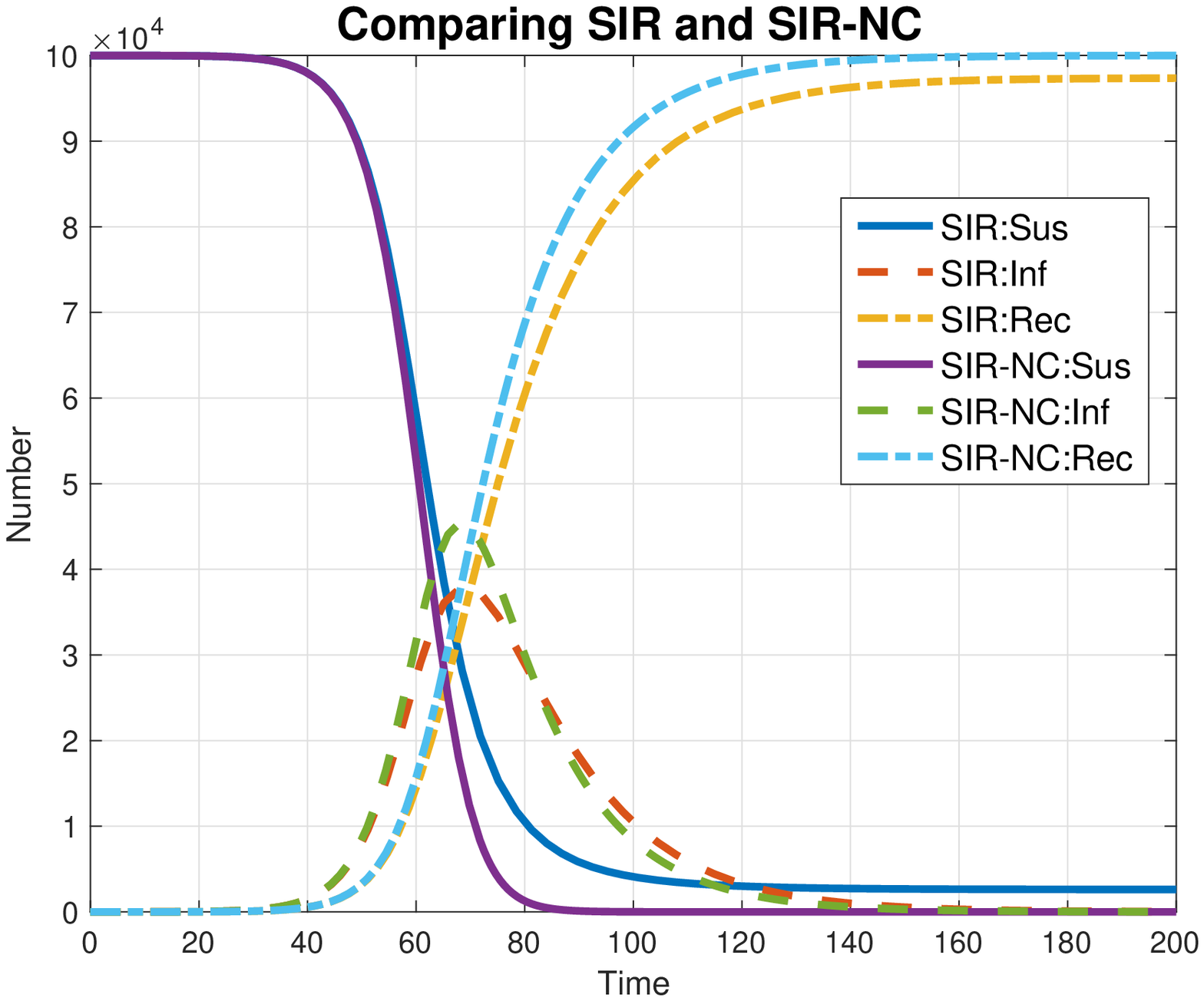}
    \caption{$N=100,000$}
    \label{plot:sir-sirnc-100k}
  \end{subfigure}
  \end{center}
  \caption{Trajectory of $I(t)$ and $S(t)$ for the SIR and SIR-NC
    systems are shown. For both systems, the other parameters are
    $\lambda=1/4,$ $\gamma=1/15,$ $I(0)=1,$ and $R(0)=0.$}
  \label{plot:SIR-SIRNC-compare} 
\end{figure}

\begin{table}[tbp]
  \begin{center}
    \begin{tabular}{|l|c|c|c|c|c|c|}
      \hline
      $N \downarrow$ $(\lambda,\gamma)\rightarrow$ 
      & \multicolumn{2}{|c|}{(0.1,0.05)}
      & \multicolumn{2}{|c|}{(0.2,0.05)}
      & \multicolumn{2}{|c|}{(0.2,0.1)}\\
      \cline{2-7} 
      & SIR & SIR-NC & SIR & SIR-NC & SIR & SIR-NC \\
      \hline
      1,000 & 135 &  138 & 54 & 53 & 68 & 69 \\
      10,000 & 181 & 184 & 70 & 69 & 91 & 92 \\
      100,000 &228 & 230 & 85 & 84 & 114 & 115 \\
      \hline
    \end{tabular}
  \end{center}
  \caption{Table shows $T_{\max}$ different combinations of $N,$
    $\lambda,$ and $\gamma$ in the SIR and SIR-NC systems. In all of
    these we assume $I(0)=1$ and $R(0)=0.$ }
  \label{tbl:Tmax-Imax}
\end{table}

\begin{figure}[tbp]
  \begin{center}
    \includegraphics[width=0.7\textwidth]{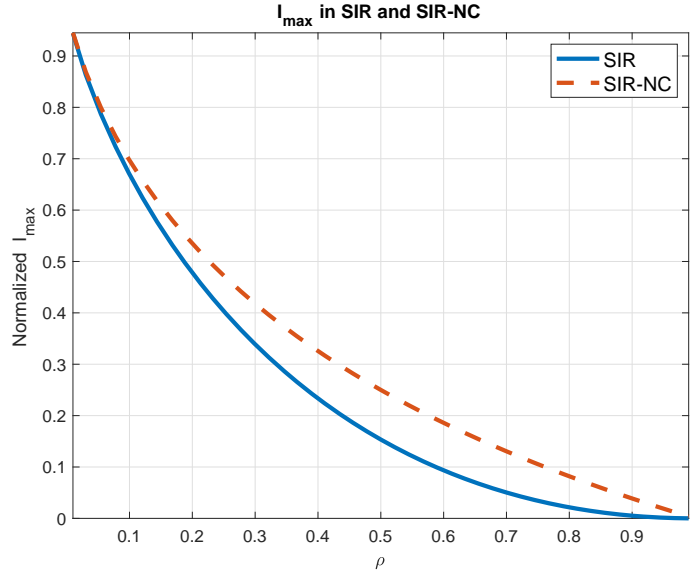}
  \end{center}
  \caption{The normalized peak infections, $I_{\max}/N,$ as a function
    of $\rho=\frac{\gamma}{\lambda}$ for the SIR and SIR-NC models for
    the case when $I(0)=1$ and $S(0)=N-1.$ $I_{\max}$ is proportional
    to $N$ and is a function of $\rho.$}
  \label{plot:sirnc-Imax-compare} 
\end{figure}  
We now discuss how the key differences between the SIR and the SIR-NC
models affect the trajectory of $I(t)$ for the same parameter
values. Recall that in the SIR model of \eqref{eq:SIR}, there is no
removal and the entire population of $N$ interact with each other. In
SIR-NC for the same $\lambda,$ the $N(t)$ is a non-increasing function
of $t$ and the growth of $I(t)$ is faster in the beginning.  This can
be seen in Fig.~\ref{plot:SIR-SIRNC-compare} that compares the
trajectory of $I(t)$ and $S(t)$ for the two models. It is also
instructive to compare $I_{\max}$ and $T_{\max}$ for the two models;
Fig.~\ref{plot:sirnc-Imax-compare} compares the $\frac{I_{\max}}{N}$
for the SIR and the SIR-NC models as a function of
$\rho=\gamma/\lambda.$ As expected, $I_{\max}$ in SIR-NC is higher
than that in SIR. Table~\ref{tbl:Tmax-Imax} lists $T_{\max}$ for the
SIR and SIR-NC systems for different values of $N,$ $\lambda,$ and
$\gamma.$ Interestingly, we see that $T_{\max}$ in both the systems
are very nearly the same. We performed these computations for several
more parameters and this observation appears to be more generally
true.

In the rest of the paper, we will only consider the SIR-NC model.

\begin{figure}[tbp]
  \begin{center}
    \includegraphics[width=0.7\textwidth]{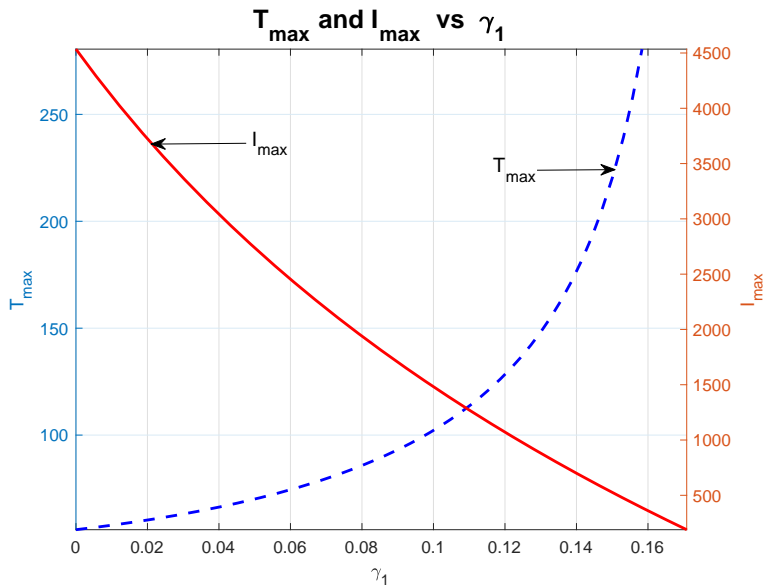}
  \end{center}
    \caption{$T_{max}$ and $I_{\max}$ are shown as a function of
    $\gamma_1$ where $\gamma_1$ corresponds to the rate infected are
    being tested and removed from the population. The other parameters
    are $N=10,000$ $\lambda=1/4,$ $\gamma=1/15+\gamma_1,$ $I(0)=1,$
    and $R(0)=0.$}
  \label{plot:SIR-NC-TImax} 
\end{figure}

In Fig.~\ref{plot:SIR-NC-TImax} we show how $I_{\max}$ and $T_{\max}$
behave if the infected could be removed at a faster rate than the
natural removal rate. The marginal relative increase in $T_{\max}$ is
significantly lower than the marginal decrease in $I_{\max}$ for small
$\gamma_1.$ In the example, observe that doubling $T_{\max}$ (from
about 55 to 110) requires $\gamma_1 > 0.11,$ i.e., about 11\% of those
infected should be removed via testing. However, halving $I_{\max}$
(from about 4500 to about 2250) requires only $\gamma_1=0.07.$

\begin{figure}[tbp]
  \begin{center}
    \begin{subfigure}{0.32\textwidth}
      \centering
      \includegraphics[width=\textwidth]{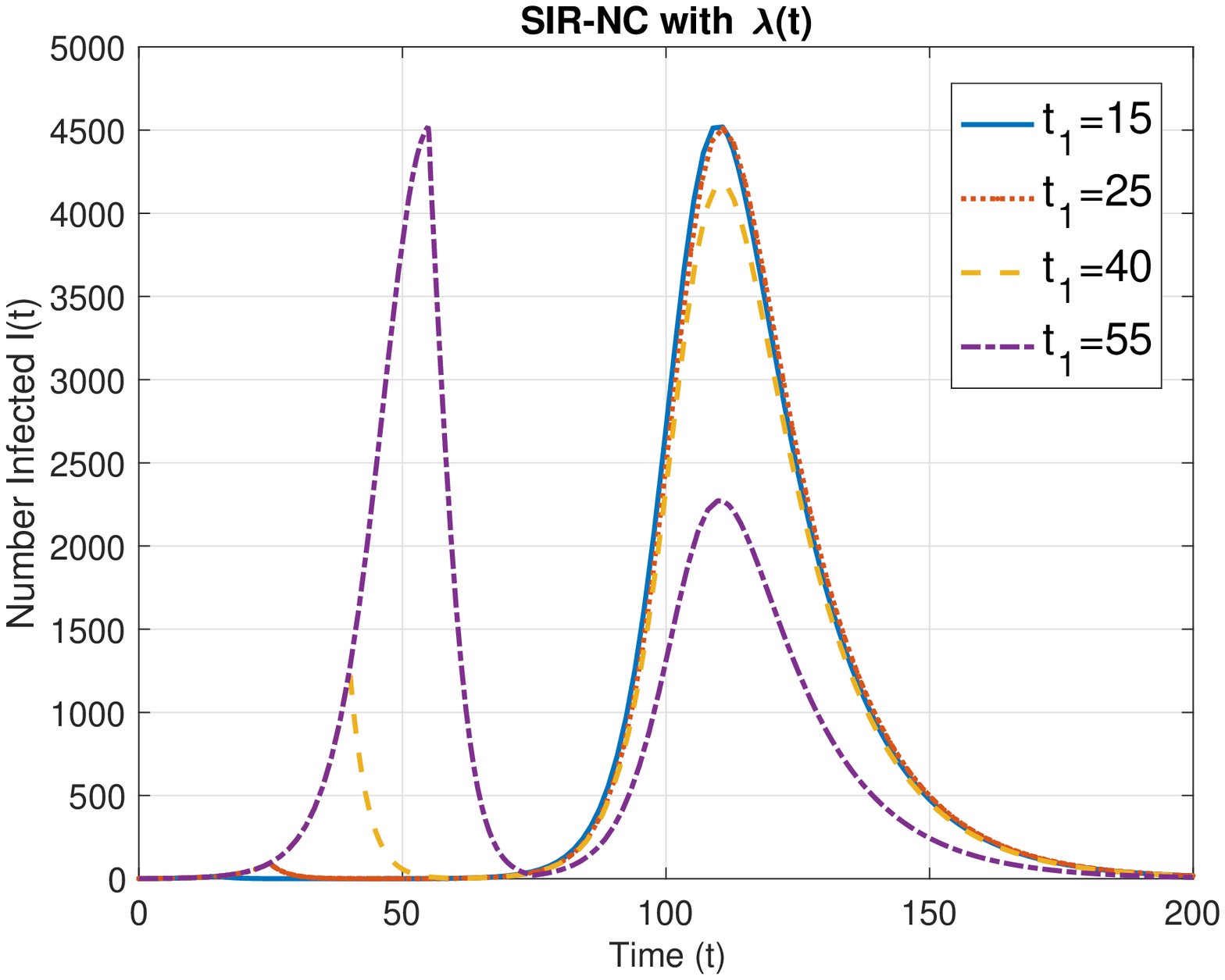}
      \caption{$\lambda(t),$ fixed $\gamma$}
    \end{subfigure}
    \begin{subfigure}{0.32\textwidth}
      \centering
      \includegraphics[width=\textwidth]{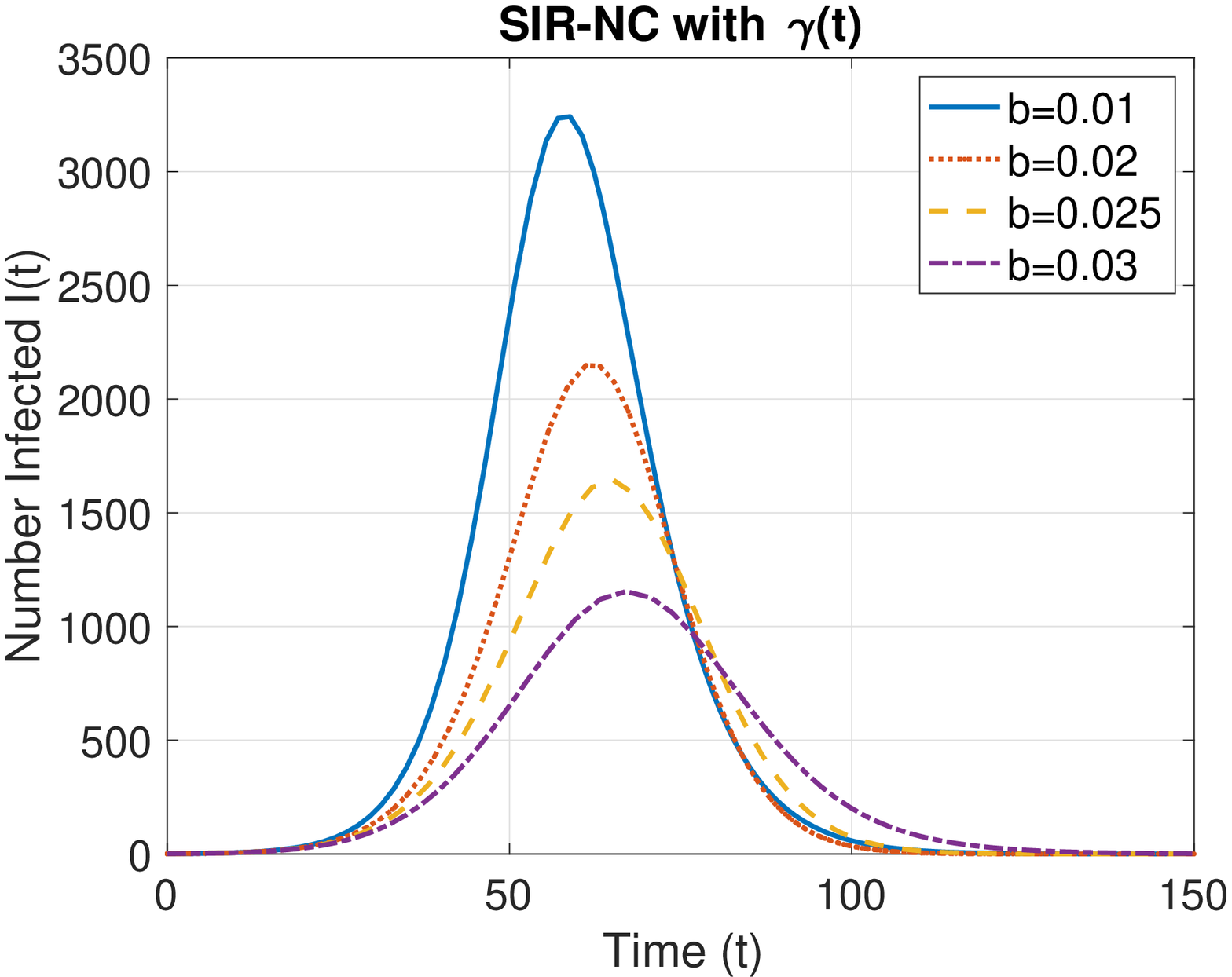}
      \caption{$\gamma(t),$ fixed $\lambda$}
    \end{subfigure}
    \begin{subfigure}{0.32\textwidth}
      \centering
      \includegraphics[width=\textwidth]{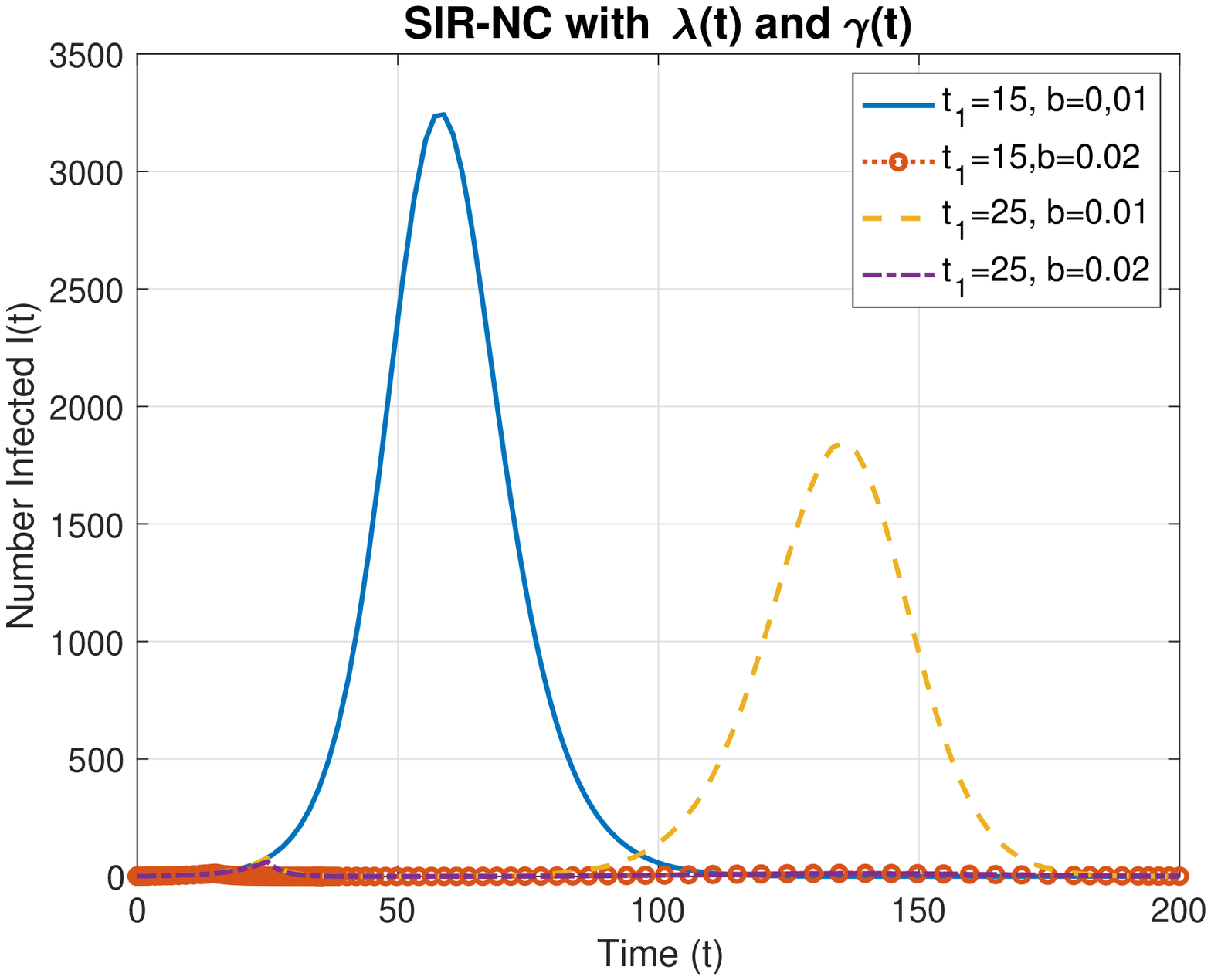}\\
      \caption{$\lambda(t)$ and $\gamma(t)$}
    \end{subfigure}
  \end{center}
  \caption{Evolution of $I(t)$ for the SIR-NC model when $\lambda,$
    $\gamma$ or both are functions of time. We have used $a=0.5$ and
    $T=20.$ The other parameters are $N=10,000$ $\lambda=1/4,$
    $\gamma=1/15,$ $I(0)=1,$ and $R(0)=0.$}
  \label{plot:SIR-NC-lambda-gamma-t} 
\end{figure}

We now illustrate the case of time varying $\lambda$ and $\gamma.$
Observe that if we assume that $\lambda(t)$ and $\gamma(t)$ are
piecewise constant, taking values $(\lambda_i, \gamma_i)$ in intervals
$(t_{i-1},t_i),$ $I(t)$ can be obtained by using
\eqref{eq:SIR-NC-results1}.

We now consider more general cases that may be applicable in epidemic
conditions. We argue later that the following explicit form of
$\lambda(t)$ would be of interest.
\begin{align*}
  \lambda(t) = \begin{cases}
    \lambda & \mbox{for $t>0$ and $t\notin (t_1, t_1+T)$}\\
    a \lambda & \mbox{$t\in (t_1, t_1+T)$}
  \end{cases}
\end{align*}
where $0 < a \leq 1.$ This models a limited lockdown period during
which there would be a significant reduction in the contact rate and
hence the infectivity. Further, we will use $\gamma(t)=
\gamma_0(1+bt)$ which models a super-exponential increase in testing
capacity and isolation.  In Fig.~\ref{plot:SIR-NC-lambda-gamma-t}a,
$\gamma(t)$ is a constant ($\gamma_0$) and $\lambda(t)$ has the form
above. We plot $I(t)$ for different values of $t_1.$ We see that an
early lockdown (corresponding to $t_1=15$ and 25) only delays the
$T_{\max}$ but does not change $I_{\max}.$ If the lockdown is started
close to the peak, it results in `second wave' with the second peak
being lower than the first. A moderate delay in the start of the
lockdown ($t_1=40$) arrests the first growth but does not
significantly affect $I_{\max}$ or $T_{\max}.$ If the lockdown starts
close to the peak, $I(t)$ comes down sharply but causes a second peak!
In Fig.~\ref{plot:SIR-NC-lambda-gamma-t}b we plot $I(t)$ keeping
$\lambda(t)$ constant, i.e., no lockdowns, and use $\gamma(t)$ as
above. This is to be expected because, the lockdowns only postpone the
inevitable. bserve that even a small amount of increased testing and
quarantining can be very effective in reducing $I_{\max}$
significantly; a 1\% increase in the rate of removal per unit time
(doubling testing capacity in 70 time units) reduces $I_{\max}$ by
about 30\% whereas a 3\% (double capacity in 24 days) can decrease
$I_{\max}$ by 70\%. $T_{\max}$ is not affected significantly in this.
Finally, in Fig.~\ref{plot:SIR-NC-lambda-gamma-t}c, we plot $I(t)$
when there is lockdown and also the $\gamma$ is increased. We see that
a lockdown and increased testing can bring the epidemic under control.

The preceding motivates us to propose formal optimal control
frameworks. They are detailed in Section~\ref{sec:SIR-NC-Control}.

\subsection{Proof of Theorem~1}
\label{sec:Thm1-proof}
Letting
\begin{eqnarray*}
  x(t) & := & \frac{S(t)}{S(t) + I(t)}, \\
  y(t) & := & \frac{I(t)}{S(t) + I(t)}
\end{eqnarray*}
we have $x(t) + y(t) \equiv 1.$ Also, $x(0) < 1$ (by the assumption that
$I(0) > 0$). We
can rewrite $\dot{x}(t)$ and  $\dot{y}(t)$ as follows. 
\begin{eqnarray*}
  \dot{x}(t) &=& \frac{\dot{S}(t)}{S(t) + I(t)} -
  \frac{S(t)(\dot{S}(t) + \dot{I}(t))}{(S(t) + I(t))^2} \\
  &=& -\lambda x(t)y(t) - x(t)(-\lambda x(t)y(t) +\lambda x(t)y(t) -
  \gamma y(t)) \\
  &=& (\gamma - \lambda)x(t)y(t) \\
  &=& (\gamma - \lambda)x(t)(1 - x(t)),\\
  && \ \\
  \dot{y}(t) &=& \frac{\dot{I}(t)}{S(t) + I(t)} -
  \frac{I(t)(\dot{S}(t) + \dot{I}(t))}{(S(t) + I(t))^2} \\
  &=& \lambda x(t)y(t) - \gamma y(t) - y(t)(-\lambda x(t)y(t) +\lambda
  x(t)y(t) - \gamma y(t)) \\ 
  &=& \lambda x(t)y(t) - \gamma y(t) + \gamma y(t)^2 \\
  &=& (\lambda - \gamma)y(t)(1 - y(t)).
\end{eqnarray*}

Then $\dot{x}(t) + \dot{y}(t) = \frac{d}{dt}(x(t) + y(t)) = 0$ as
expected. Now
\begin{align*}
  && \dot{x}(t) & = \  (\gamma - \lambda)x(t)(1 - x(t))\\
  & \Longrightarrow \hspace{0.3in}  &
  \frac{dx(t)}{x(t)(1 - x(t))} & =  (\gamma - \lambda)dt \\
  & \Longrightarrow \hspace{0.3in} & 
  dx(t)\left(\frac{1}{x(t)} + \frac{1}{1 - x(t)}\right) & = \ (\gamma
  - \lambda)dt \\
  & \Longrightarrow \hspace{0.3in} &
  \log\left(\frac{x(t)}{1 - x(t)}\right) & = \ c + (\gamma - \lambda)t
  \hspace{0.3in}  \mbox{for some} \ c \in \mathcal{R} \\
  & \Longrightarrow \hspace{0.3in} &
  \frac{x(t)}{1 - x(t)} & = \ Ce^{(\gamma - \lambda)t}, \ C := e^c \\
  & \Longrightarrow \hspace{0.3in} &
  x(t) & = \ \frac{Ce^{(\gamma - \lambda)t}}{1 + Ce^{(\gamma -
      \lambda)t}} \\ 
  & & y(t) &= 1 - x(t) \ = \ \frac{1}{1 + Ce^{(\gamma - \lambda)t}}
\end{align*}

For $\gamma < \lambda$, $x(t) \to 0$ and $y(t) \to 1$ as expected,
i.e., for $x(0) \in [0, 1)$, $(0, 1)$ is the asymptotically stable
equilibrium.  It also follows by setting $t = 0$ that $x(0) =
\frac{C}{1 + C}$, so that $C = \frac{x(0)}{ 1 - x(0)} =
\frac{S(0)}{I(0)}.$

We are now ready to derive the expression for $I(t).$ Letting $N(t) =
S(t) + I(t)$, we have 
\begin{eqnarray*}
  \dot{N}(t) &=& -\gamma I(t), \\
  \dot{I}(t) &=& I(t)\left(\lambda \frac{S(t)}{N(t)} - \gamma\right)\\
  &=& I(t)(\lambda x(t) - \gamma)  
 \end{eqnarray*}
Then clearly 
\begin{align*}
  I(t) & = \ I(0)e^{\int_0^t(\lambda x(s) - \gamma)ds} \nonumber \\
  & = I(0)e^{\int_0^t\left(\frac{\lambda Ce^{-(\lambda - \gamma)s}}{1
      + Ce^{-(\lambda - \gamma)s}} - \gamma\right)ds} \nonumber \\
  & = I(0)\left(\frac{1 + Ce^{(\gamma - \lambda)t}}{1 +
    C}\right)^\frac{\lambda}{(\gamma - \lambda)} \nonumber \\
  S(t) & = S(0)e^{-\int_0^t\left(\frac{\lambda }{1 + Ce^{-(\lambda -
        \gamma)s}}\right)ds} \nonumber \\
  & = S(0) \left(\frac{e^{(\lambda - \gamma)t} + C}{1 +
      C}\right)^{\frac{\lambda}{(\gamma - \lambda)}}.\nonumber
\end{align*}
Clearly, $I(t) > 0$ $\forall$ $t \geq 0.$ Thus we also see that
$\dot{N}(t) < 0$, implying $N(t)\downarrow N(\infty) \geq 0$. Then
$I(t) = y(t)N(t) \to 0$.  Also, from the explicit expression for
$S(t)$, we have $S(t) \to 0$, thus $N(\infty) = 0$.

To determine $T_{\max},$ note that initially $x(t) \approx 1$. Coupled
with $\lambda > \gamma$, we have $\dot{I}(t) > 0$, so that $I(t)$
increases till it hits the isocline $\frac{S}{N} = 1 -
\frac{\gamma}{\lambda}$ at time $T_{\max}.$ Thus $T_{\max}$ can be
calculated as follows. We have
\begin{eqnarray*}
  \dot{I}(t) &=& I(t)(\lambda x(t) - \gamma) \\ &=&
  I(t)\left(\frac{\lambda Ce^{(\gamma-\lambda)t}}{1 + Ce^{(\gamma -
      \lambda)t}} - \gamma\right).
\end{eqnarray*}
Thus $T_{\max}$ is the solution of the equation
\begin{align*}
  && C e^{(\gamma-\lambda)T_{\max}} &= \left(1 + Ce^{(\gamma -
    \lambda)T_{\max}}\right)\frac{\gamma}{\lambda} \\
  & \Longrightarrow \hspace{0.0in} & 
  \left(\frac{\lambda - \gamma}{\lambda}\right)Ce^{-(\lambda -
    \gamma) T_{\max}} &= \frac{\gamma}{\lambda} \\
  & \Longrightarrow \hspace{0.0in} &
  T_{\max} & = \frac{\left(\log\left[C\left(\frac{(\lambda -
        \gamma)}{\gamma}\right)\right]\right)^+}{ \lambda - \gamma}.
\end{align*}

This completes the proof of the constant coefficient case. The proof
for time-dependent $\lambda, \gamma$ goes exactly the same way. \hfill
$\qed$

Now consider
\begin{equation*}
  \dot{R}(t) = \beta I(t)
\end{equation*}
for some $\beta < \gamma$. Thus $\beta$ is the recovery rate and
$\gamma - \beta$ the death rate. Then
\begin{equation*}
  R(t) \uparrow R^\infty(\lambda, \gamma) := \beta\int_0^\infty I(s)ds
  < \infty.
\end{equation*}

\section{Further Variations}
\label{sec:variations}

We now consider several variations of the basic SIR-NC model that we
introduced in the previous section. Specifically, we present two types
of variations and characterize the solutions. The first set of
variations is motivated by the actions of several communities during
the COVID-19 pandemic. Communities are isolating themselves to
different degrees and also applying differing degrees of social
distancing rules within the community. To analyze this, in
Section~\ref{sec:SIR-NC-Import} we consider the case when a community
has `imported infections', i.e., an external agency introduces
infections into the community at a steady rate. In
Section~\ref{sec:SIR-NC-communities} we use the results of
Section~\ref{sec:SIR-NC-Import} to analyze the interaction two
separate communnities that interact with each other. We obtain
explicit characterizations for both these models. The second variation
is to introduce natural births and deaths in
Section~\ref{sec:birth-deaths} where we analyze the equilibrium
behaviour of the model.

\subsection{SIR-NC with imported infections}
\label{sec:SIR-NC-Import}

In many communities, infections are also induced by interactions with
other communities. Assume that each of the susceptible in the
population meets $\mu_e$ persons from outside of the community and
$p_e$ of those meetings result in an infection. Defining $\nu= p_e
\mu_e,$ we have the following SIR-NC model in which infections are
also imported from other communities.
\begin{eqnarray*}
  \frac{dS(t)}{dt} & = & - \lambda \frac{I(t) S(t)}{N(t)} - \nu S(t) \\
  \frac{dI(t)}{dt} & = & \lambda \frac{I(t) S(t)}{N(t)} + \nu S(t) -
  \gamma I(t) \\  
  \frac{dR(t)}{dt} & = & \gamma I(t) \\
  N(t) & = & S(t) + I(t) 
\end{eqnarray*}

We observe that the condition for an epidemic to break out is
\begin{displaymath}
  \lambda \frac{I(0)}{N(0)} + \nu > \gamma \frac{I(0)}{S(0)}
\end{displaymath}

Define $x(t) := \frac{S(t)}{N(t)}$. Along the lines of
Section~\ref{sec:Thm1-proof} we can show that $x(t)$ satisfies the
following o.d.e.
%
\begin{align*}
  \dot{x}(t) & = (\gamma - \lambda)x(t)(1 - x(t)) - \nu x(t) \\
   & =  -(\lambda - \gamma)x(t)(1 + \frac{\nu}{\lambda - \gamma} -
  x(t))
\end{align*}
This can be written as
\begin{align*}
  && \left(\frac{\lambda - \gamma}{\lambda - \gamma +
    \nu}\right)dx(t)\left(\frac{1}{x(t)} + \frac{1}{1 +
    \frac{\nu}{\lambda - \gamma} - x(t)}\right) & = -(\lambda -
  \gamma)dt \\
  & \Longrightarrow \hspace{0.1in} &
  \log\left(\frac{x(t)}{1 + \frac{\nu}{\lambda - \gamma} -
    x(t)}\right) & = c - (\lambda - \gamma + \nu)t.
    \end{align*}
    That is,
    \begin{eqnarray*}
      x(t) = \left(\frac{\lambda - \gamma + \nu}{\lambda -
        \gamma}\right)\frac{C_1e^{-(\lambda - \gamma + \nu)t}}{1 +
        C_1e^{-(\lambda - \gamma + \nu)t}}
    \end{eqnarray*}
for $C_1 = e^c$. Setting $t = 0$, we get 
\begin{displaymath}
  C_1 = \frac{x(0)}{\frac{\lambda - \gamma + \nu}{\lambda - \gamma} -
    x(0)} \ .
\end{displaymath}
Now,
\begin{eqnarray*}
  y(t) &=& 1 - x(t) \\
  &=& \frac{(\lambda - \gamma) - \nu C_1e^{-(\lambda - \gamma +
      \nu)t}}{(\lambda - \gamma)(1 + C_1e^{-(\lambda - \gamma +
      \nu)t})}.
\end{eqnarray*}
Thus
\begin{align*}
  && \dot{S}(t) &= -\lambda S(t)(1 - x(t)) - \nu S(t) \\  
  && &= -S(t)((\lambda + \nu) - \lambda x(t)) \\  
  && &= -S(t)\left((\lambda + \nu) - \frac{\lambda(\lambda - \gamma +
    \nu)}{\lambda - \gamma}\left(\frac{C_1e^{-(\lambda - \gamma +
      \nu)t}}{1 + C_1e^{-(\lambda - \gamma + \nu)t}}\right)\right) \\
  & \Longrightarrow & 
  S(t) &= S(0)e^{-(\lambda + \nu)t + \frac{\lambda(\lambda - \gamma +
      \nu)}{\lambda - \gamma}\int_0^t\frac{Ce_1^{-(\lambda - \gamma +
        \nu)s}}{1 + C_1e^{-(\lambda - \gamma + \nu)s}} ds} \\
  && &= S(0)e^{-(\lambda + \nu)t - \frac{\lambda}{\lambda -
      \gamma}\log\left( \frac{1 + C_1e^{-(\lambda - \gamma +
        \nu)t}}{1+C_1}\right)} \\
  && &= S(0)\left(\frac{1 + C_1e^{-(\lambda - \gamma +
      \nu)t}}{1+C_1}\right)^{-\frac{\lambda}{\lambda -
      \gamma}}e^{-(\lambda + \nu)t}.
\end{align*}
\remove{
  \begin{align*}
  && \dot{I}(t) & = I(t)(\lambda x(t) + \nu S(t) - \gamma I(t)) \\
  & \Longrightarrow &
  I(t) &= I(0) \exp\Big\{-\gamma t - \frac{\lambda}{\lambda -
    \gamma}\log\left(1 + C_1e^{-(\lambda - \gamma + \nu)}t\right) - \\
  &&& \ \ \ \nu S(0)\int_0^te^{-(\lambda + \nu)s}\left(1 +
  C_1e^{-(\lambda - \gamma + \nu)s}\right)^{- \frac{\lambda}{\lambda -
      \gamma}}ds\Big\} \\
  && & = I(0)\left(1 + C_1e^{-(\lambda - \gamma +
    \nu)t}\right)^{\frac{\lambda}{\lambda - \gamma}}e^{-\gamma
    t}\times \\
  && & \ \ \ e^{\nu S(0)\int_0^t\left(1 + C_1e^{-(\lambda - \gamma +
      \nu)s}\right)^{-\frac{\lambda}{\lambda - \gamma}}e^{-(\lambda +
      \nu)s}ds}.
\end{align*}
}

Then
\begin{align*}
  && \dot{I}(t) & = I(t)\lambda x(t) + \nu S(t) - \gamma I(t) \\
  &&            & = I(t)(\lambda x(t) -\gamma) + \nu S(t) \\
  & \Longrightarrow & I(t) & \ = \ I(0)\exp\left(\int_0^t(\lambda x(s)
  - \gamma)ds\right) + \nu\int_0^t \exp\left(\int_s^t(\lambda x(y) -
  \gamma) dy \right) S(s)ds.
\end{align*}
We have just derived the following replacement for
Theorem~\ref{thm:SIR-NC} for the case of SIR-NC with imported
infections. $T_{\max}$ for this model can be obtained as in
Theorem~\ref{thm:SIR-NC}.  

\begin{theorem}
  The solution to the SIR-NC system with imported infections is as
  follows.
  \begin{align}
    S(t) &= S(0)\left(\frac{1 + C_1e^{-(\lambda - \gamma +
        \nu)t}}{1+C_1}\right)^{-\frac{\lambda}{\lambda -
        \gamma}}e^{-(\lambda + \nu)t}. \nonumber \\
    I(t) & = I(0) \left(
    \frac{1+C_1 e^{-(\lambda-\gamma+\nu)t}}{1+C_1}
    \right)^{-\frac{\lambda}{\lambda-\gamma}}
    e^{- \gamma t} \nonumber \\
    & \hspace{0.3in} + \nu \int_0^t \left( \left(\frac{1+C_1
    e^{-(\lambda-\gamma+\nu)t}}{1+C_1
    e^{-(\lambda-\gamma+\nu)s }} 
    \right)^{- \frac{\lambda}{\lambda-\gamma}}
    e^{ - \gamma (t - s)} \right) S(s) \ ds
    \label{eq:SIR-NC2-results}
  \end{align}
  where 
  \begin{displaymath}
    C_1 = \frac{S(0)(\lambda - \gamma)}{I(0) (\lambda - \gamma + \nu)
      + S(0) \nu }
  \end{displaymath}
  \label{thm:SIR-NC-Import}
\end{theorem}


\begin{figure}[tbp]
  \begin{subfigure}{0.47\textwidth}
    \centering
    \includegraphics[width=\textwidth]{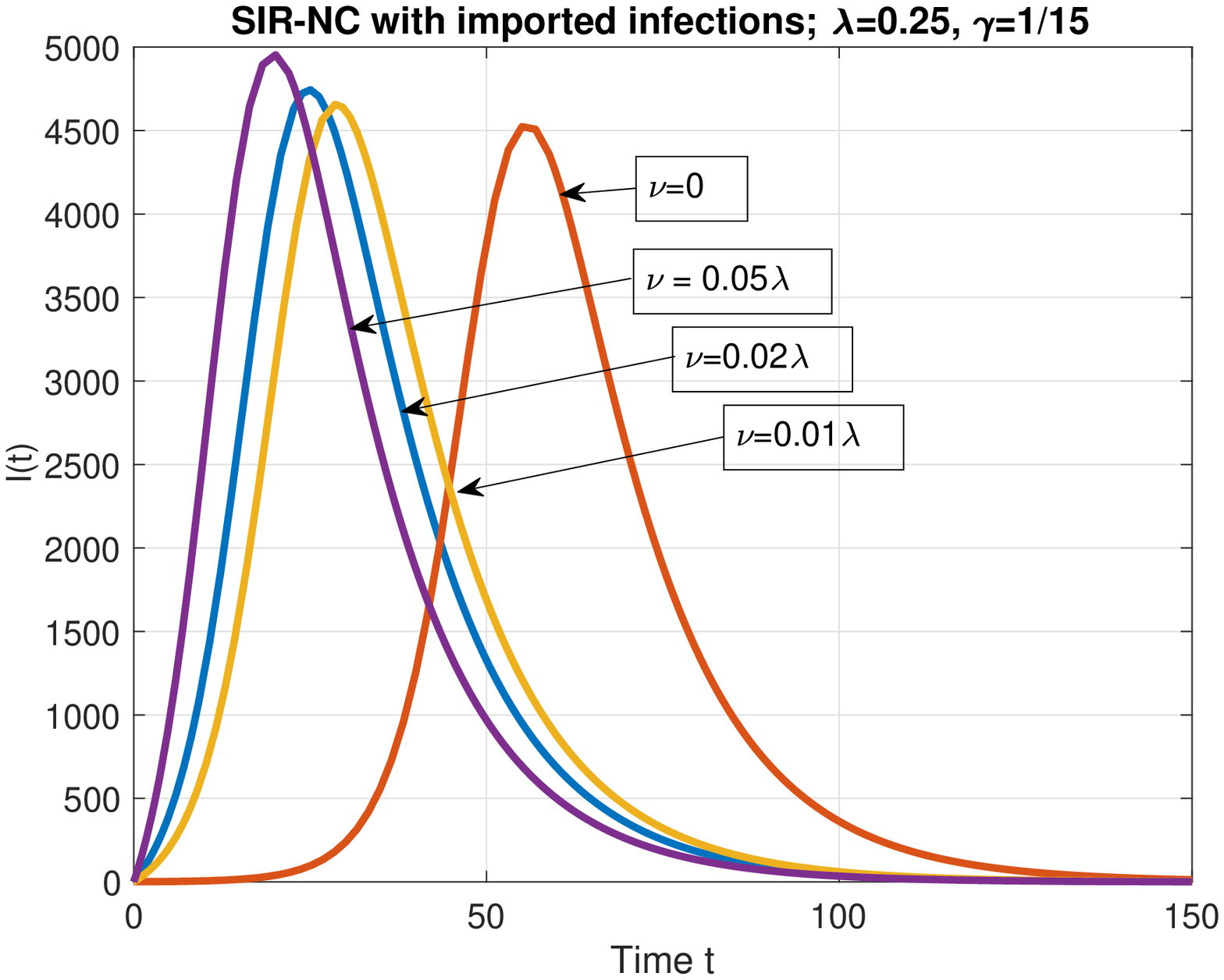}
    \caption{$I(t)$ for different values of $\nu.$ }
    \label{plot:SIR-NC-import1}
  \end{subfigure}
  \begin{subfigure}{0.47\textwidth}
    \centering
    \includegraphics[width=\textwidth]{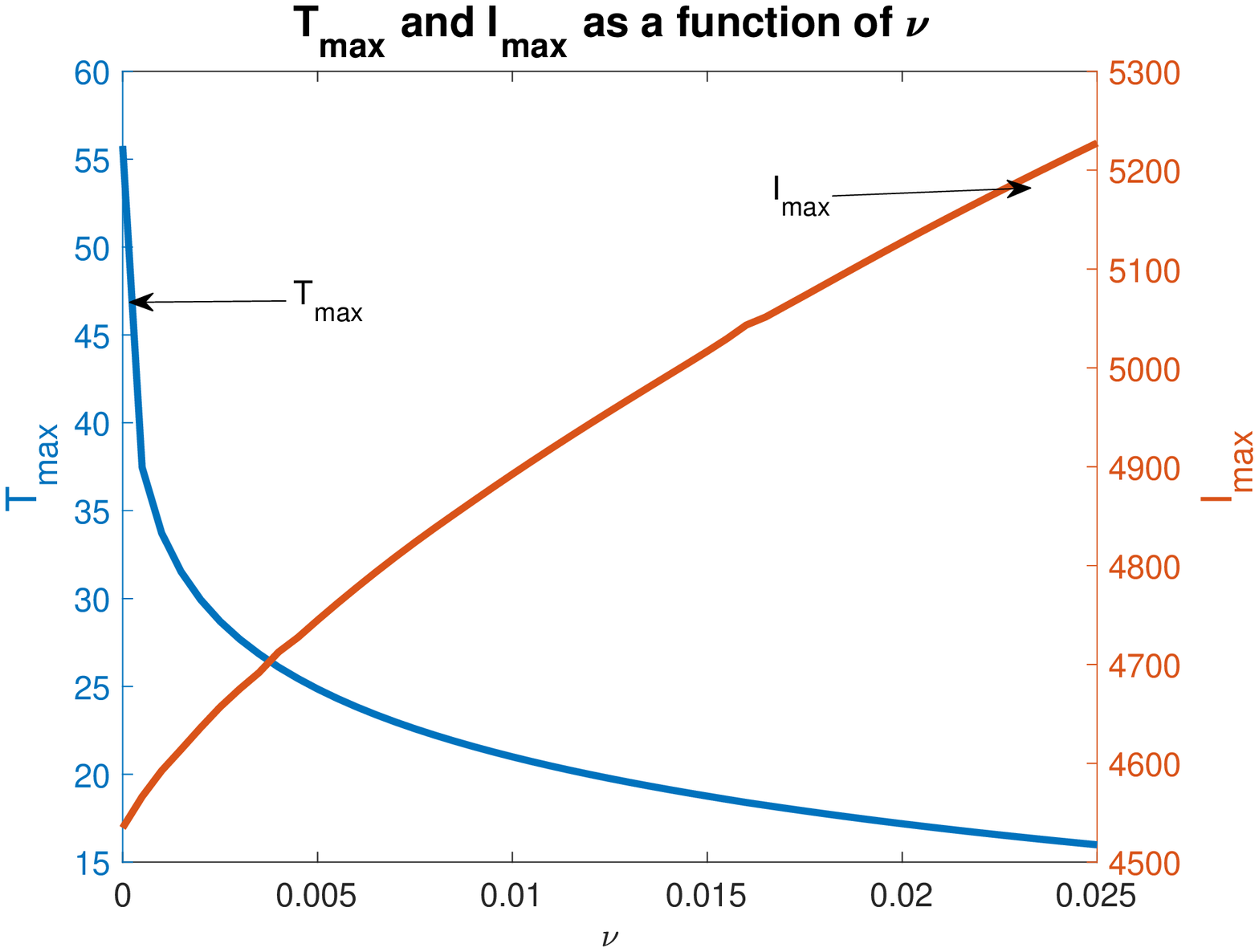}
    \caption{$T_{\max}$, $I_{\max}$ as a function of $\nu.$ }
    \label{plot:SIR-NC-import2}
  \end{subfigure}
  \caption{Illustrating the effect of imported infections. We have
    used $\lambda=1/4$ and $\gamma=1/15m,$ $I(0)=1,$ and $S(0)=9999.$
  }
\end{figure}

\subsubsection{Numerical Examples}
We now illustrate this model with some numerical
examples. Fig.~\ref{plot:SIR-NC-import1} illustrates $I(t)$ for
different values of $\nu.$ Observe that even a small $\nu$ advances
$T_{\max}$ significantly, $\nu=0.01\lambda$ reduces $T_{\max}$ by
nearly half (from about 55 to about 29). This dependence is studied in
more detail in Fig.~\ref{plot:SIR-NC-import2} that shows $T_{\max}$
and $I_{\max}$ as a function of $\nu.$ The marginal reduction in
$T_{\max}$ is significant for small values of $\nu$ but $I_{\max}$ is
less dramatic.

\subsection{SIR-NC in interacting communities}
\label{sec:SIR-NC-communities}

Now consider two interacting communities, labeled $a$ and $b,$ that
have different intra-community contact and infection infection rates,
and in addition, inter-community contact and infection rates.  Let
$N_a$ and $N_b$ be the size of the two communities. Defining
$\lambda_a,$ $\lambda_b$, $\gamma_a$, $\gamma_b$ and $\lambda_{ab}$ in
the obvious way, we have the following equations defining the system
evolution.
\begin{eqnarray}
  \frac{dS_a(t)}{dt} & = & - S_a(t) \left( \lambda_a
  \frac{I_a(t)}{N_a(t)} + \lambda_{ab} \frac{I_b(t)}{N_b(t)}
  \right) \label{a1} \\
  \frac{dI_a(t)}{dt} & = & S_a(t)\left( \lambda_a
  \frac{I_a(t)}{N_a(t)} + \lambda_{ab} \frac{I_b(t)}{N_b(t)} \right) -
  \gamma_a I_a(t) \label{a2} \\
  \frac{dR_a(t)}{dt} & = & \beta_a I_a(t) \label{a3} \\
  \frac{dS_b(t)}{dt} & = & - S_b(t) \left( \lambda_b
  \frac{I_b(t)}{N_b(t)} + \lambda_{ab} \frac{I_a(t)}{N_a(t)} \right)
  \label{b1}   \\  
  \frac{dI_b(t)}{dt} & = & S_b(t) \left( \lambda_b
  \frac{I_b(t)}{N_b(t)} + \lambda_{ab} \frac{I_a(t)}{N_a(t)} \right) -
  \gamma_b I_b(t) \label{b2} \\  
  \frac{dR_b(t)}{dt} & = & \beta_b I_b(t) \label{b3} \\  
  N_a & = & S_a(t) + I_a(t)  \label{a4} \\
  N_b & = & S_b(t) + I_b(t) \label{b4}
\end{eqnarray}
Letting
\begin{align*}
  & x_a(t) := \frac{S_a(t)}{N_a(t)}, \ y_a(t) := 1 - x_a(t) =
  \frac{I_a(t)}{N_a(t)}, \\
  & x_b(t) := \frac{x_b(t)}{N_b(t)}, \ y_b(t) := 1 - x_b(t) =
  \frac{I_b(t)}{N_b(t)},
\end{align*}
we have the coupled dynamics
\begin{eqnarray*}
  \dot{x}_a(t) &=& f_a(x_a(t), x_b(t)) := - (\lambda_a -
  \gamma_a)x_a(t)(1 - x_a(t)) - \lambda_{ab}x_a(t)(1 - x_b(t)), \\  
  \dot{x}_b(t) &=& f_b(x_a(t), x_b(t)) := - (\lambda_b -
  \gamma_b)x_b(t)(1 - x_b(t)) - \lambda_{ab}x_b(t)(1 - x_a(t)),  
\end{eqnarray*}
and an analogous dynamics for $(y_a(\cdot), y_b(\cdot))$. This is a
cooperative dynamics in the sense of \cite{Hirsch85} and by virtue of
having bounded trajectories, will converge for almost all initial
conditions to the set of its equilibria. These can be found by setting
the right hand side equal to zero, and their stability can be analyzed
using local linearization. We do not pursue this here.  This equation,
however, is not analytically tractable. So we aim for an approximate
solution by assuming that $\lambda_{ab} = \epsilon,$ $0 < \epsilon <<
\frac{1}{T},$ where $T > 0$ is the time horizon of interest. Let
$S^0_i(t), I^0_i(t), i = a,b,$ denote the decoupled dynamics
corresponding to $\epsilon = 0$. Then for general $\epsilon > 0$, we
can use the results of the preceding section to obtain an
approximation for $S_a(t), I_a(t), T^a_{\max}$ ($:=$ the corresponding
peaking time) as follows.

However, the two are coupled. We can decouple them by
resorting to the approximations 
\begin{displaymath}
  \frac{I_i(t)}{N_i(t)} \approx \frac{I_i^0(t)}{N_i^0(t)}, \ i = 1,2.
\end{displaymath}
In fact, $S_i(t) = S^0_i(t) + O(\epsilon T), I_i(t) = I_i^0(t) +
O(\epsilon T), i = 1,2, t \in [0, T]$. An exact expression for this
error term can be given in terms of Alekseev's nonlinear variation of
constants formula \cite{Alekseev61}, \cite{Brauer66} as follows.

View the dynamics \eqref{a1})-\eqref{a3}, \eqref{a4} of community $a$
(a symmetric argument works for community $b$) as a perturbation of
the uncoupled dynamics
\begin{eqnarray}
  \dot{S}(t) &=& -\lambda\frac{S(t)I(t)}{S(t) + I(t)}, \label{nil1} \\
  \dot{I}(t) &=& \lambda\frac{S(t)I(t)}{S(t) + I(t)} - \gamma
  I(t), \label{nil2} \\  
  \dot{R}(t) &=& \beta I(t). \label{nil3}
\end{eqnarray}
The linearization of (\ref{nil1})-(\ref{nil2}) around the nominal
trajectory $(S(\cdot), I(\cdot), R(\cdot))$ is given as
follows. Letting $S_a(t) \approx S(t) + \ep s(t)$, $I_a(t) \approx
I(t) + \ep i(t)$, we have 
\begin{eqnarray}
  \dot{s}(t) &=&\left(-\lambda \frac{I(t)^2}{(S(t) +
    I(t))^2}\right)s(t) + \left(-\lambda \frac{S(t)^2}{(S(t) +
    I(t))^2}\right)i(t), \label{lin1} \\
  \dot{i}(t) &=& \left(\lambda \frac{I(t)^2}{(S(t) +
    I(t))^2}\right)s(t) + \left(\lambda \frac{S(t)^2}{(S(t) + I(t))^2}
  - \gamma\right)i(t).
  \label{lin2}
\end{eqnarray}
Write this as
\begin{displaymath}
  \dot{q}(t)  = A(t)q(t), \ t \geq t_0,
\end{displaymath}
where $q(t) := [s(t), i(t)]^T$ and $A(t)$ is defined in the obvious
manner. Let
\begin{displaymath}
  (S(t,t_0; S_0,I_0), I(t,t_0;S_0,I_0)), \ t \geq t_0,
\end{displaymath}
denote the solution for (\ref{nil1})-(\ref{nil2}) for $t \geq t_0$
with $S(t_0) = S_0, I(t_0) = I_0$. Define 
\begin{displaymath}
  (S(t,t_0; S_0,I_0,R_0), I(t,t_0;S_0,I_0,R_0), R(t,t_0;
  S_0,I_0,R_0)), \ t \geq t_0,
\end{displaymath}
as above, but with (\ref{a1})-(\ref{a3}) replacing
(\ref{nil1})-(\ref{nil2}) and $S^\ep(t_0) = S_0, I^\ep(t_0) = I_0,
R^\ep(t_0) = R_0$. Also, let $\Phi(t,t_0; S_0,I_0), t \geq t_0$ denote
the fundamental solution of (\ref{lin1})-(\ref{lin2}), i.e., solution
to the matrix linear system
\begin{displaymath}
  \dot{\Phi}(t,t_0; S_0,I_0) = A(t)\Phi(t,t_0; S_0,I_0), \ t \geq t_0;
  \ \Phi(t_0,t_0; S_0,I_0) = I.
\end{displaymath}
Let 
\begin{eqnarray*}
  X(t,t_0; S_0,I_0) &:=& [S(t,t_0; S_0,I_0), I(t,t_0; S_0,I_0)]^T,\\
  X^\ep(t,t_0; S_0,I_0,R_0) &:=& [S(t,t_0; S_0,I_0,R_0),
    I(t,t_0;S_0,I_0,R_0)]^T,
\end{eqnarray*}
and  
\begin{displaymath}
  \varepsilon(t) := \left[ \ \lambda_{ab}\frac{I_b(t)}{N_b(t)} ,
    \ \lambda_{ab}\frac{I_b(t)}{N_b(t)} \ \right].
\end{displaymath}
Then by Alekseev's nonlinear variation of constants formula
\cite{Alekseev61}, \cite{Brauer66},
\begin{eqnarray}
  \lefteqn{X^\ep(t,t_0; S_0,I_0,R_0) = X(t,t_0; S_0,I_0) \ +} \nonumber \\
  && \int_{t_0}^t\Phi(t,s; S^\ep(s,t_0; S_0,I_0,R_0), I^\ep(s,t_0;
  S_0,I_0,R_0))\varepsilon(s)ds. \label{Alek1} 
\end{eqnarray}
This is an exact expression for the net perturbation of the solution,
albeit implicit. The second term on the right is $O(\ep)$. Thus if we
ignore the $O(\ep^2)$ terms, we can write
\begin{eqnarray}
  \lefteqn{X^\ep(t,t_0; S_0,I_0,R_0) \approx X(t,t_0; S_0,I_0) \ +}
  \nonumber \\
  && \int_{t_0}^t\Phi(t,s; S(s,t_0; S_0,I_0), I(s,t_0;
  S_0,I_0))\varepsilon(s)ds. \label{Alek1a}
\end{eqnarray}
This gives an explicit expression for the net perturbation of the
solution, albeit only approximate to within $O(\ep^2)$.  Note,
however, that this approximation will be valid only for a finite time
interval, because the hypothesis that $\varepsilon(t) = O(\ep)$ may
not continue to hold unless $T$ is small.  For computational purposes,
the above procedure can be repeated over such small time intervals.

\textbf{Remark:} The case of $\lambda_{ab} \neq \lambda_{ba}$ is also
of interest, e.g., when one community practices the safety
recommendations more strictly than the other community. This case can
be analyzed analogously.

\begin{figure}[tbp]
  \begin{subfigure}{0.4\textwidth}
    \centering
    \includegraphics[width=\textwidth]{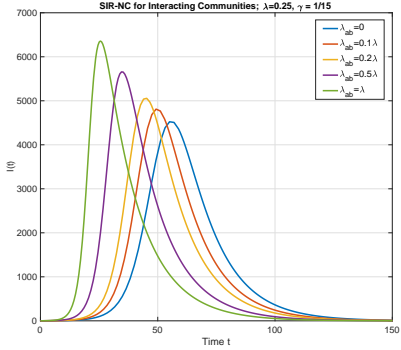}
    \caption{The symmetric case with $N_a=N_b=10,000,$
      $\lambda_a=\lambda_b=\lambda=1/4,$ $\gamma_a=\gamma_b=1/15,$ and
    $\lambda_{ab}=\lambda_{ba}.$}
    \label{plot:SIR-NC-community-symm}
  \end{subfigure}
  \hfill 
  \begin{subfigure}{0.4\textwidth}
    \centering
    \includegraphics[width=\textwidth]{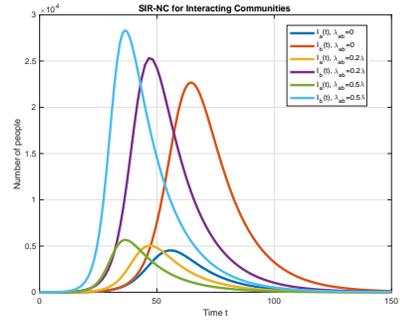}
    \caption{Asymmetry in size with $N_a=5,000, N_b=10,000,$
      $\lambda_a=\lambda_b=\lambda=1/4,$ $\gamma_a=\gamma_b=1/15,$ and
      $\lambda_{ab}=\lambda_{ba}.$ }
    \label{plot:SIR-NC-community-size-asymm}
  \end{subfigure}

  \begin{subfigure}{0.4\textwidth}
    \centering
    \includegraphics[width=\textwidth]{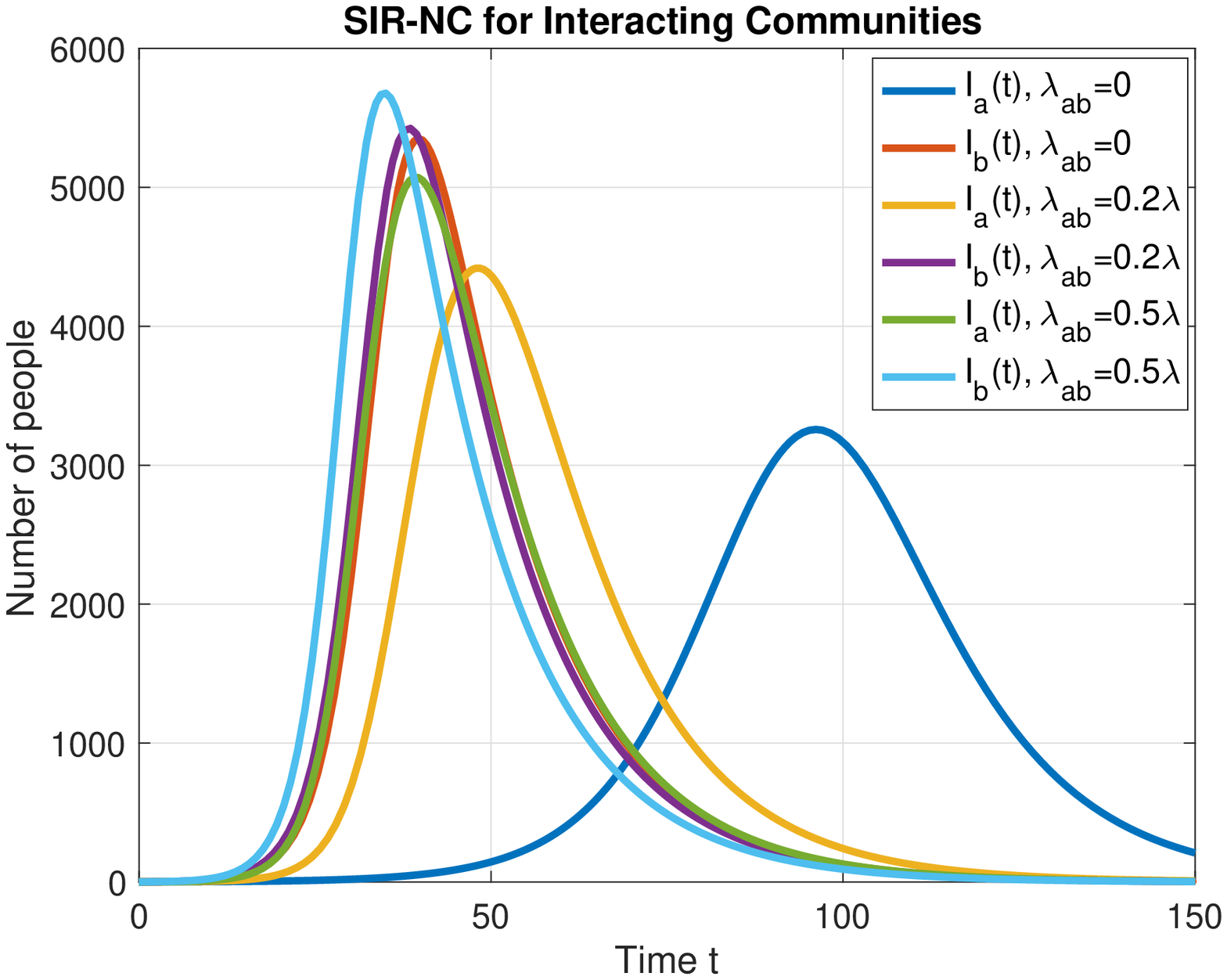}
    \caption{Asymmetry in $\lambda$ with $N_a=N_b=10,000,$
      $\lambda_a=\lambda=1/6,$ $/\lambda_b=\lambda/2,$
      $\gamma_a=\gamma_b=1/15,$ and $\lambda_{ab}=\lambda_{ba}.$ }
        \label{plot:SIR-NC-community-lambda-asymm}
  \end{subfigure}
  \hfill 
  \begin{subfigure}{0.4\textwidth}
    \centering
    \includegraphics[width=\textwidth]{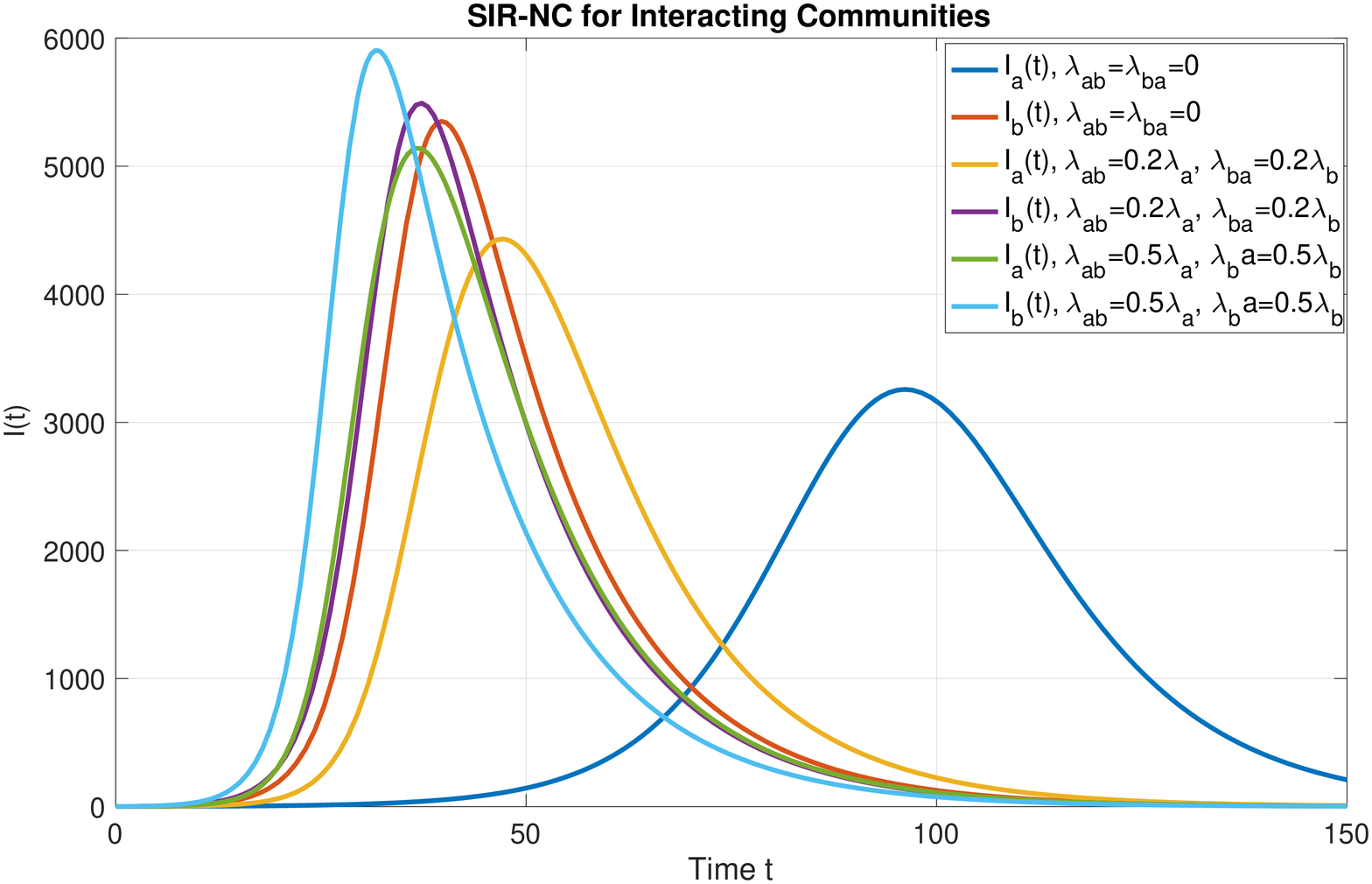} 
    \caption{Asymmetry in $\lambda$ and size. $N_a=5,000, N_b=10,000,$
      $\lambda_a=\lambda=1/6,$ $\lambda_b=\lambda/2,$ and
      $\gamma_a=\gamma_b=1/15.$ }
    \label{plot:SIR-NC-community-both-asymm}
  \end{subfigure}

  \caption{Two interacting communities. The effect of $\lambda_{ab}$
    and $\lambda_{ba}$ on $I_a(t)$ and $I_b(t).$}
  \label{plot:SIR-NC-community}  
\end{figure}

\begin{figure}[tbp]
  \centering
  \includegraphics[width=0.6\textwidth]{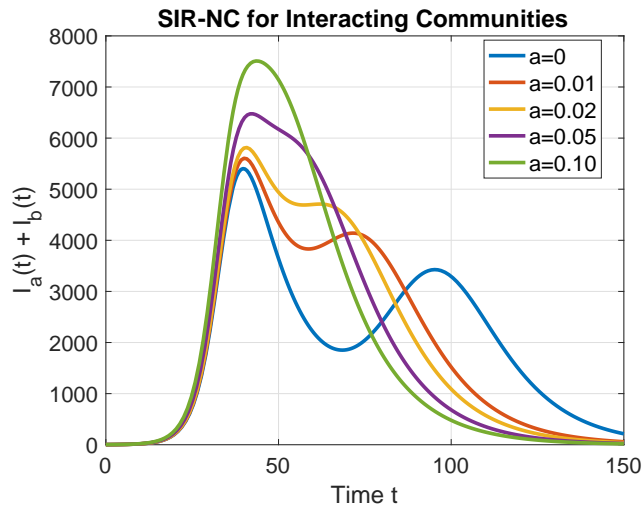} 
  \caption{$I_a(t)+I)b(t)$ with two asymmetric
    communities. $N_a=5,000, N_b=10,000,$ $\lambda_a=1/6,$ $\lambda_b
    = 1/3,$ $\lambda_{ab}=a \lambda_a,$ $\lambda_{ba}=a \lambda_b,$
    and $\gamma_a=\gamma_b=1/15.$}
  \label{plot:SIR-NC-community-total}  
\end{figure}

\subsubsection{Numerical Examples}
Fig.~\ref{plot:SIR-NC-community} illustrates the effects of the
intercommunity interaction rates, $\lambda_{ab}$ and $\lambda_{ba},$
on $I(t)$ when different kinds of communities interact. In
Fig.~\ref{plot:SIR-NC-community-symm} we consider two similar
communities with the same initial population and same $\lambda.$ The
trajectories for $I(t)$ for both communities are, as is to be
expected, identical. However, even a small amount of interaction can
significantly affect the trajectory, leading to increased $I_{\max}$
and a reduced $T_{\max}.$ This is in line with the results of the
preceding subsection. In Fig.~\ref{plot:SIR-NC-community-size-asymm},
we see that when $\lambda_a=\lambda_b$ and $\lambda_{ab} =
\lambda_{ba},$ and the sizes are significantly different, the
trajectories of $I(t)$ are very similar in both the communities and
they peak at nearly the same time. However, if $N_a=N_b$ and
$\lambda_a \neq \lambda_b$ while $\lambda_{ab}=\lambda_{ba},$ then
even a small amount of interactions matter as can be seen in the
significant shift in both $T_{\max}$ and $I_{\max}$ for community $a$
which has a lower $\lambda.$ This is seen in
Fig.~\ref{plot:SIR-NC-community-lambda-asymm}. Finally, in
Fig.~\ref{plot:SIR-NC-community-both-asymm} we show $I_{a}(t)$ and
$I_b(t)$ when there is asymmetry in the two communities both in terms
of the size and the $\lambda.$ It is more interesting to see
$I_{a}(t)+I_b(t)$ for this last case, as shown in
Fig.~\ref{plot:SIR-NC-community-total}. Observe that if the
interaction is small, the total $I(t)$ has a comparatively `fatter'
middle, possibly with a double hump. This disappears when the
interactions increase.

\subsection{Introducing births and `normal' deaths}
\label{sec:birth-deaths}

The models for long duration epidemics, like the way Covid-19 is
expected to be, should include natural births and deaths. The standard
SIR model has been extended to include natural births and deaths.
However, almost all the models continue to make the population
conservation assumption---the natural death rate is equal to the birth
rate and the total population is conserved. The SIR-NC model naturally
allows for unconstrained birth and death rates. Thus in this section
we describe the SIR-NC model that includes births and natural deaths
with no constraints on the rates. However, if we introduce some restrictions
on the parameters, not unlike those that are derived from the standard
SIR model, we can provide additional insights to the properties
of the evolution of the epidemic. Here, we make the natural assumption
that all births are susceptible.

Consider the dynamics with a small birth rate and a small death rate
for `normal' deaths (as opposed to deaths due to the epidemic), given
by:
\begin{eqnarray}
  \dot{S}^\ep(t) &=& -\lambda\frac{S^\ep(t)I^\ep(t)}{N^\ep(t)} +
  \kappa S^\ep(t) + \upsilon_1I^\ep(t) +
  \upsilon_2R^\ep(t), \label{ep1} \\
  \dot{I}^\ep(t) &=& \lambda\frac{S^\ep(t)I^\ep(t)}{N^\ep(t)} - \gamma
  I^\ep(t) - \nu_1I^\ep(t), \label{ep2} \\
  \dot{R}^\ep(t) &=& \beta I^\ep(t) - \nu_2R^\ep(t), \label{ep3}
\end{eqnarray}
where $N^\ep(t) = S^\ep(t) + I^\ep(t) + R^\ep(t)$ and $\kappa,
\upsilon_i, \nu_i = O(\ep), i = 1,2,$ for some $0 < \ep << 1$. Note 
that we have now included the recovered, i.e., $R^\ep(t)$ in the 
definition of the normalizing factor $N^\ep(t)$ in view of the fact 
that it directly affects the evolution of $S^\ep(\cdot)$ given by 
(\ref{ep1}) and can therefore no longer be viewed as having `dropped 
out'. The $\upsilon_i$'s are birth rates and the $\nu_i$'s the rates 
of normal deaths for the infected and recovered populations 
respectively, and are positive. The parameter $\kappa$ is the net rate
of birth minus normal deaths for the susceptible population, and can
be either positive or negative.

\begin{figure}[tbp]
  \begin{subfigure}{0.31\textwidth}
    \centering
    \includegraphics[width=\textwidth]{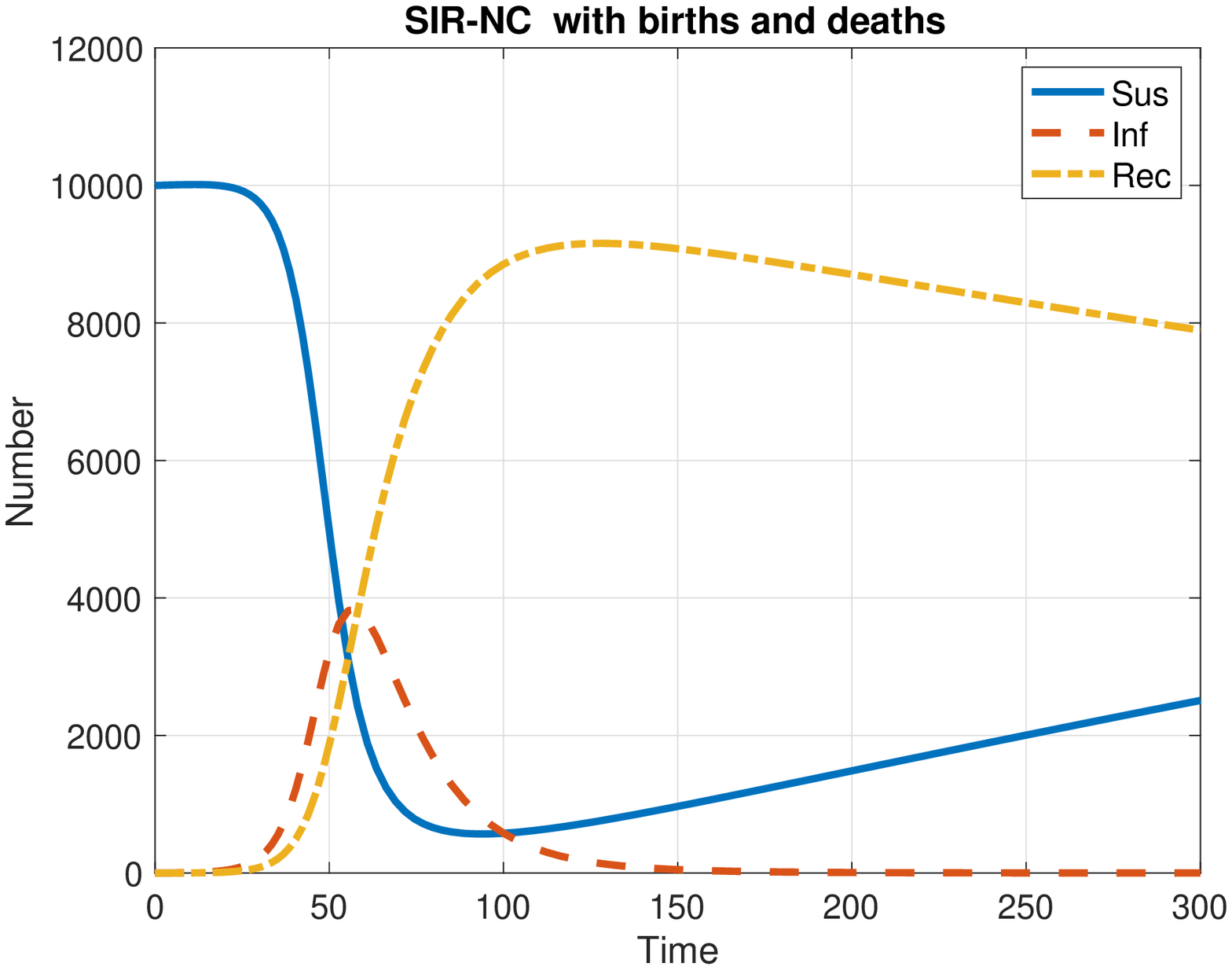}
    \caption{More natural births. $\kappa=0.0002,$
      $v_1=v_2=0.0012.$}
    \label{plot:bd-xs-300}
  \end{subfigure}
  \hfill
  \begin{subfigure}{0.31\textwidth}
    \centering
    \includegraphics[width=\textwidth]{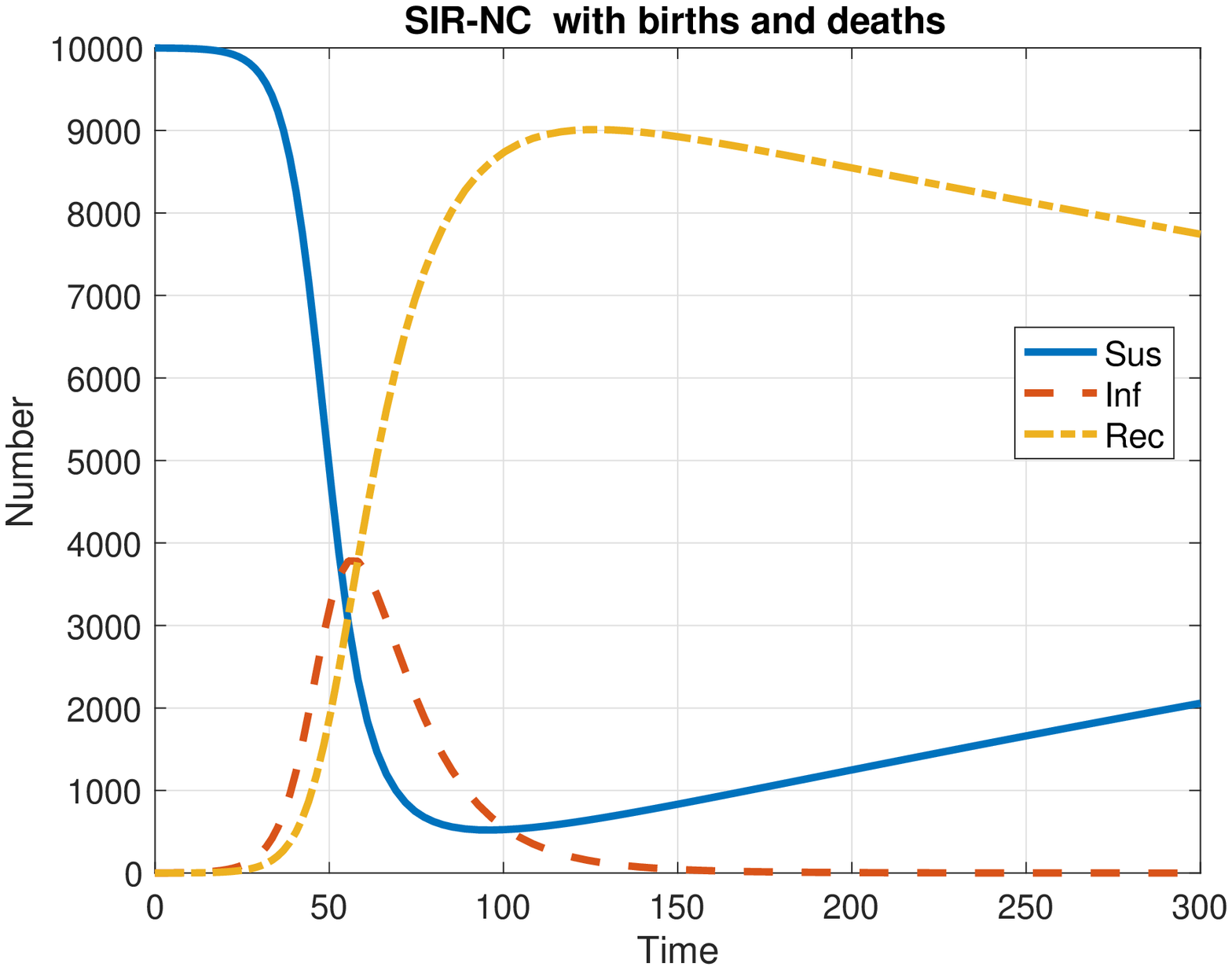}
    \caption{Balanced natural births and deaths. $\kappa=0,$
      $v_1=v_2=0.001.$ }
    \label{plot:bd-bal-300}
  \end{subfigure}
  \hfill 
  \begin{subfigure}{0.31\textwidth}
    \centering
    \includegraphics[width=\textwidth]{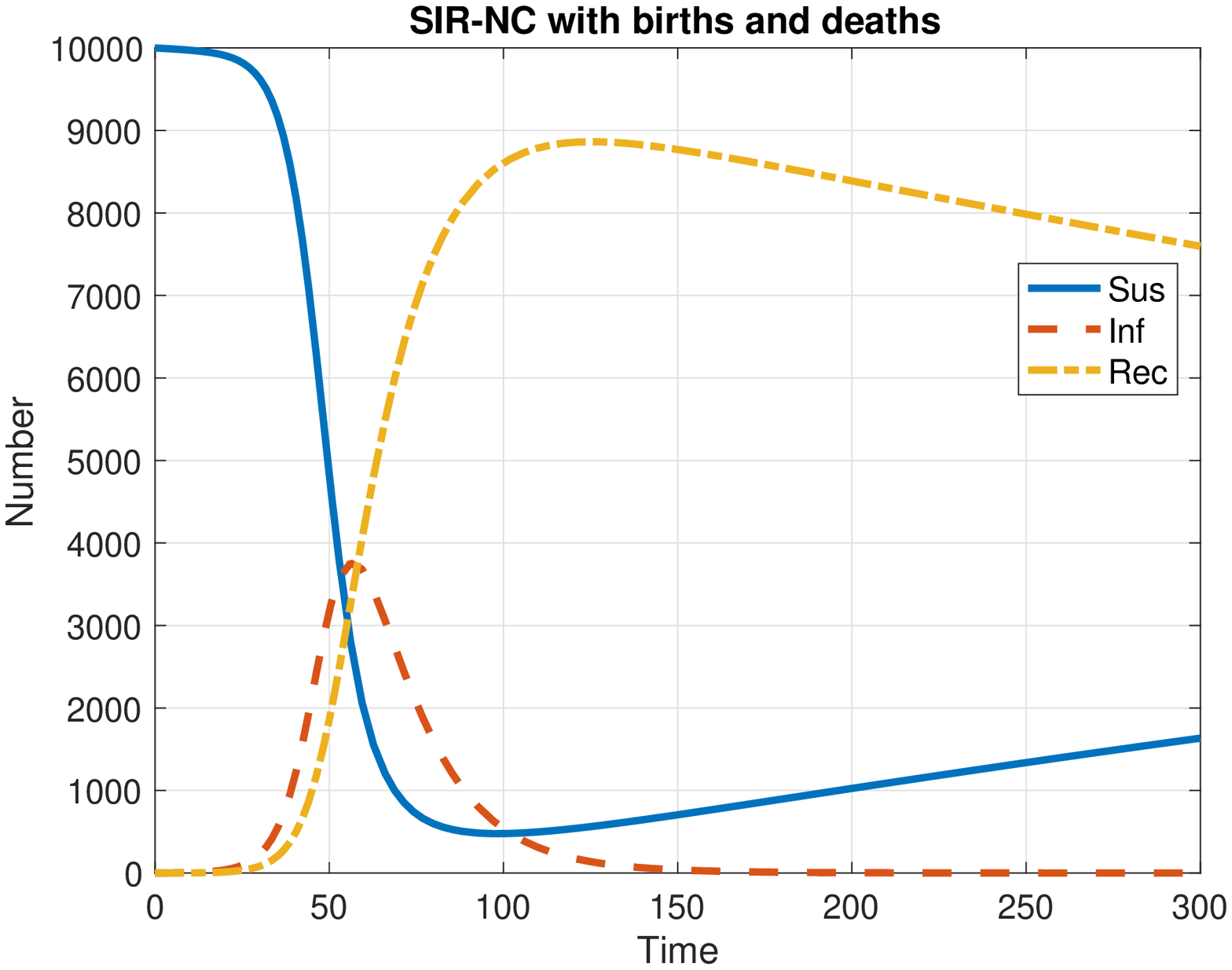}
    \caption{More natural deaths. $\kappa=-0.0002,$
      $v_1=v_2=0.0008.$}
    \label{plot:bd-less-300}
  \end{subfigure}

  \begin{subfigure}{0.31\textwidth}
    \centering
    \includegraphics[width=\textwidth]{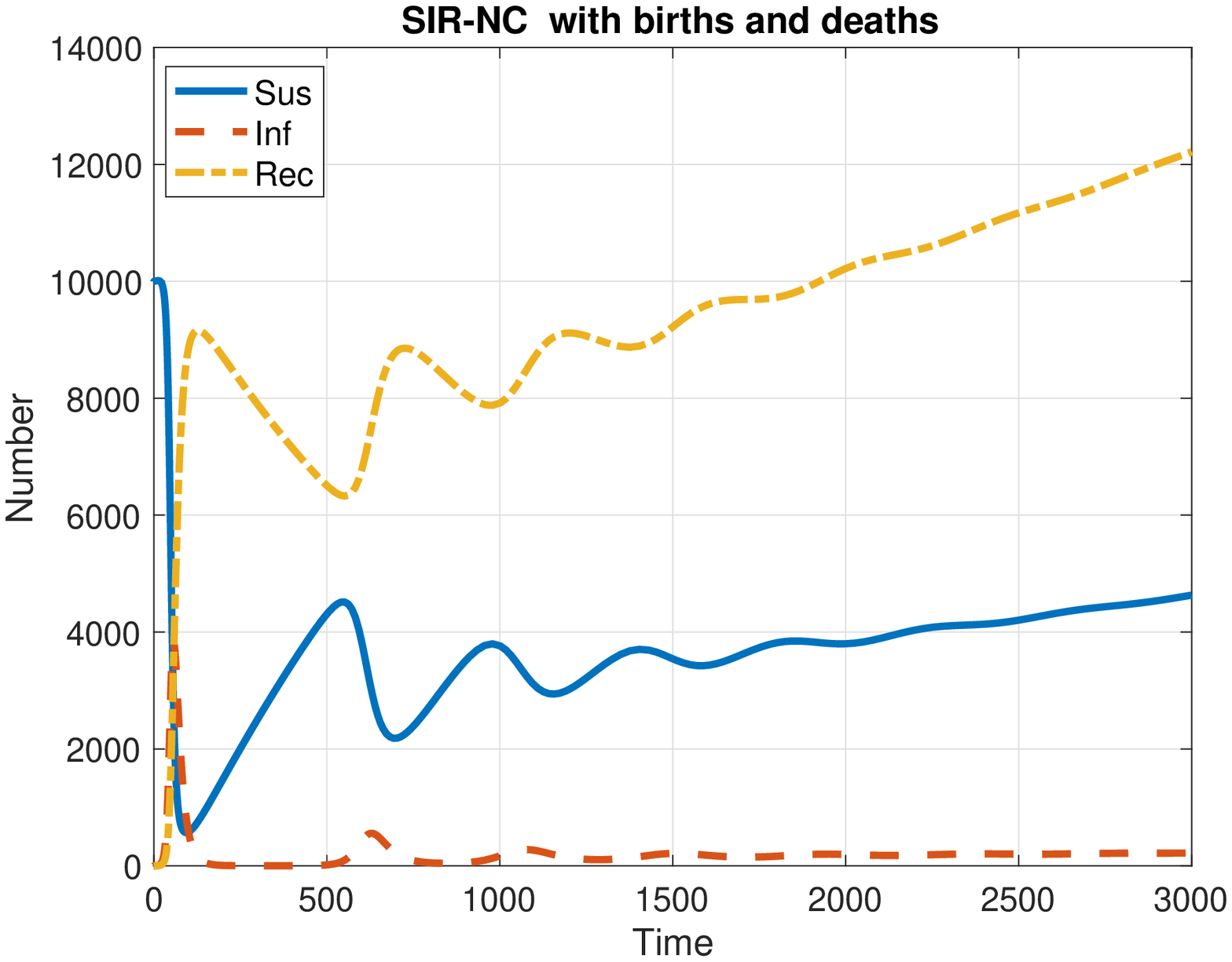}
    \caption{More natural births. $\kappa=0.0002,$
      $v_1=v_2=0.0012.$}
    \label{plot:bd-xs-3000}
  \end{subfigure}
  \hfill 
  \begin{subfigure}{0.31\textwidth}
    \centering
    \includegraphics[width=\textwidth]{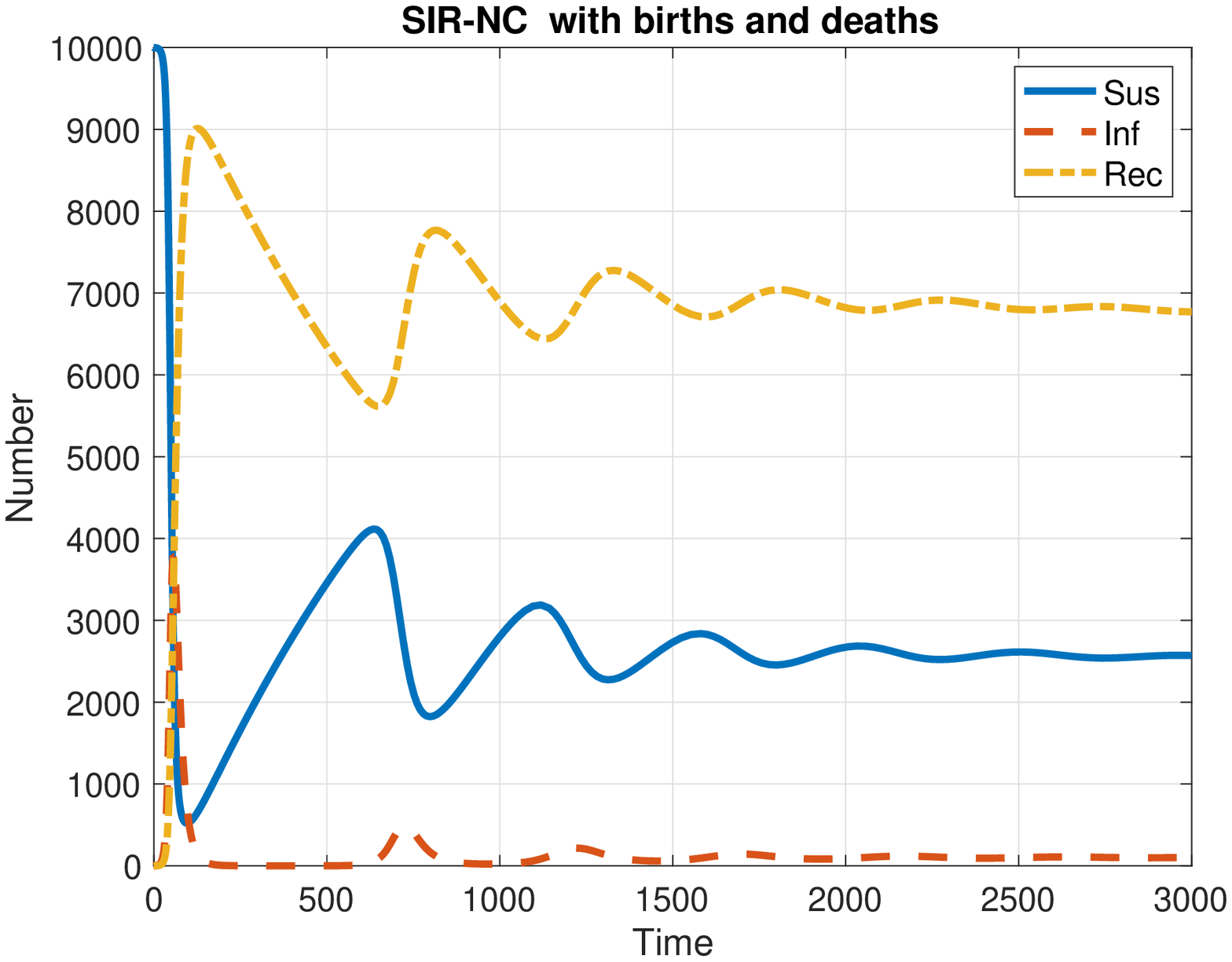}
    \caption{Balanced natural births and deaths.  $\kappa=0,$
      $v_1=v_2=0.001.$}
    \label{plot:bd-bal-3000}
  \end{subfigure}
  \hfill 
  \begin{subfigure}{0.31\textwidth}
    \centering
    \includegraphics[width=\textwidth]{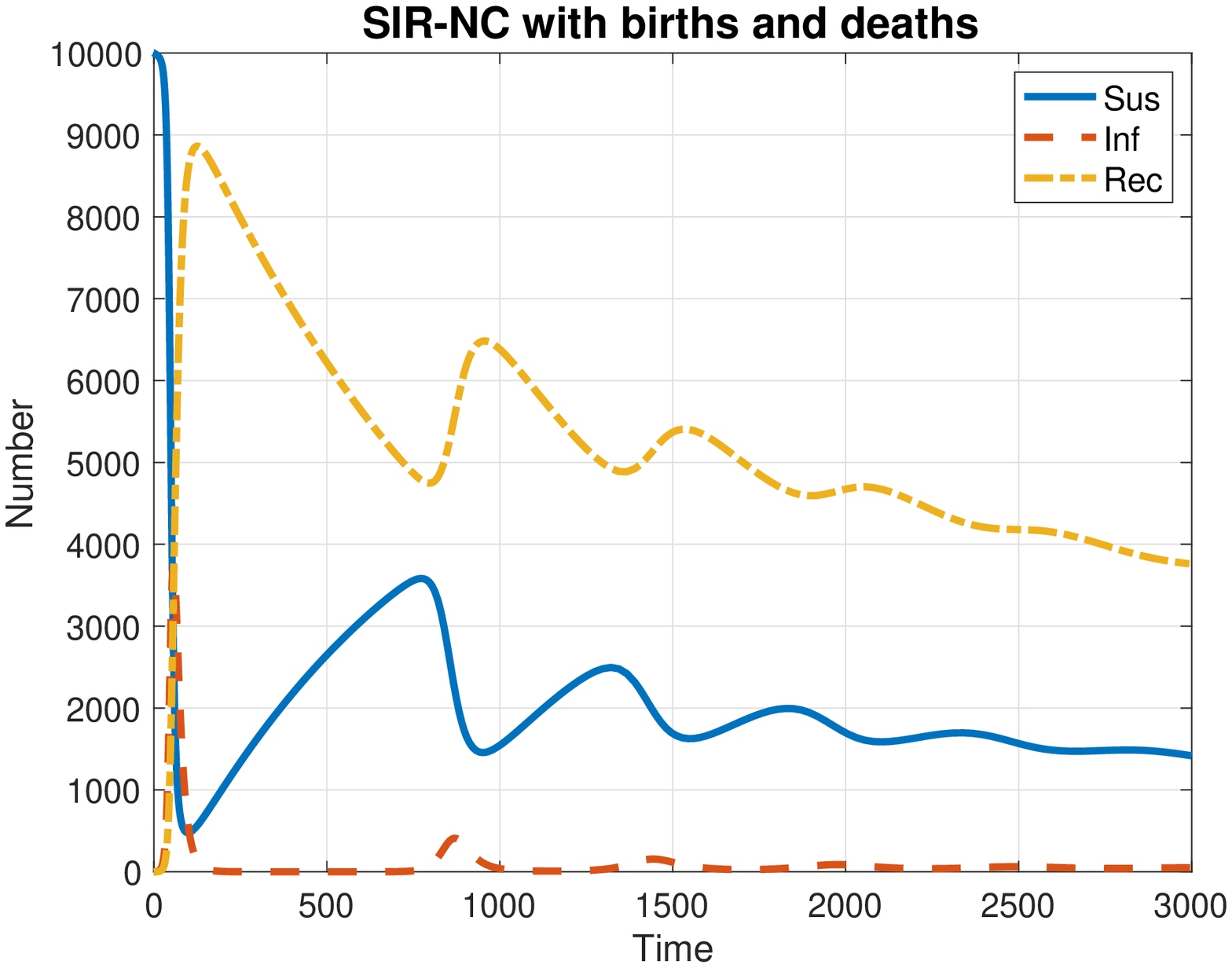}
    \caption{More natural deaths. $\kappa=-0.0002,$
      $v_1=v_2=0.0008.$}
    \label{plot:bd-less-3000}
  \end{subfigure}
  
  \caption{$S(t),$ $I(t),$ and $R(t)$ in the SIR-NC model that
    includes births and natural deaths. The top panel shows the short
    term behavior and the bottom panel shows the long term
    behavior. We have used $\lambda=1/4,$ $\gamma=1/15,$ $N=10,000,$
    $\nu_1=\nu_2=0.001,$ $\beta=0.98\gamma.$}
  \label{plot:birth-deaths}
\end{figure}

\begin{figure}
  \centering
  \includegraphics[width=0.6\textwidth]{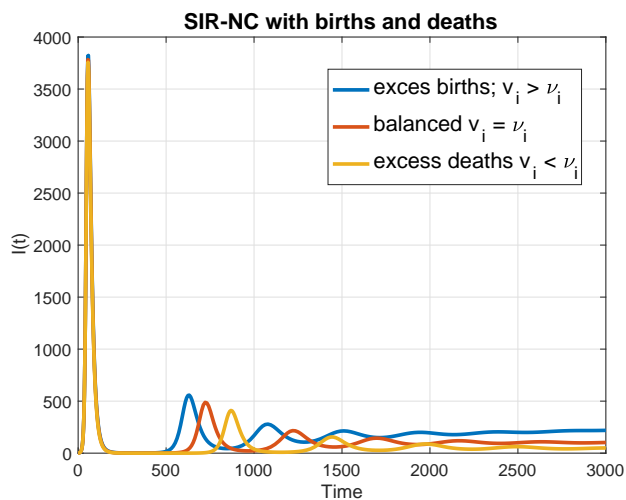}
  \caption{$I(t)$ for the three cases of Fig.~\ref{plot:birth-deaths}.}
  \label{plot:bd-I-of-t} 
\end{figure}

In Fig.~\ref{plot:birth-deaths} we show the trajectories of $S(t),$
$I(t),$ and $R(t)$ for this model for three different cases: birth
rate higher than natual death rate, balanced birth and natural death
rates, and natural death rates higher than birth rates. We show both
the short term and the long term trajectories. Observe that in all the
cases, there are significant oscillations in $S(t),$ $I(t),$ and
$R(t).$ The long term trajectory of $I(t)$ for the three cases are
shown in more detail in Fig.~\ref{plot:bd-I-of-t}. We observe that
$I(t)$ decays very slowly and $I(3000)$ has values 219, 102, and 50
for the three cases of, respectively, excess births, balanced, excess
deaths.

The preceding observations can be made more precise with the following
additional assumption: \\
\noindent \textbf{$(\dagger)$} $\kappa \leq 0, \ \upsilon_1 - \nu_1 <
\gamma, \ \upsilon_2 \leq \nu_2.$\\
This in particular includes the case of \\
\noindent \textbf{$(\dagger\dagger)$} $\kappa = 0, \upsilon_i = \nu_i$
for $i = 1,2$,\\
\noindent
corresponding to a balance of birth rate and the rate of natural
deaths community-wise. The former also includes the case of excess of
natural death rates over the birth rates. 

\begin{theorem}
  Under \textbf{$(\dagger)$}, $\hat{X}(t) :=
  (S^\ep(t),I^\ep(t),R^\ep(t))$ converges to $(0,0,0)$  for almost all initial
  conditions. 
\end{theorem}

\begin{proof} 

  Summing up the equations (\ref{ep1})-(\ref{ep3}), we have
  \begin{displaymath}
    \dot{N}^{\ep}(t) \leq -(\gamma - \beta)I^\ep(t) < 0,
  \end{displaymath}  
  since $I^\ep(0) \Longrightarrow I^\ep(t) > 0 \ \forall \ t >
  0$. Then $\N^\ep(t) \downarrow N \geq 0$. If $N = 0$, there is
  nothing to prove, so suppose $N > 0$. Then every limiting trajectory
  $(S'(\cdot),I'(\cdot),R'(\cdot))$ of (\ref{ep1})-(\ref{ep3}) in the
  $\omega$-limit set $\Omega$ of (\ref{ep1})-(\ref{ep3}) must satisfy
  $I'(\cdot) \equiv 0$. But then (\ref{ep3}) forces $R'(\cdot) \equiv
  0$, so that $S'(\cdot) \equiv N$. But $(0,0,N)$ is an unstable
  equilibrium for $N > 0$, since a small perturbation that renders
  $I'(0) > 0$ will make the trajectory move away from it. The claim
  follows.
\end{proof}

Therefore a necessary condition for non-extinction is that at least
oneof the  sub-communities should have a non-zero value for rate of birth minus
the rate of normal deaths, i.e., $\kappa>0,$ or $v_1 > \nu_1,$
or $v_2 > \nu_2.$ The most interesting aspect of this analysis is the
case \textbf{$(\dagger\dagger)$}, which subsumes the conditions under
which the classical SIR model is derived. Recall that in the classical SIR model,
non-extinction is possible. For SIR-NC, however, even a small death rate due to
the epidemic leads to extinction. 

More generally, we can say the following.

\begin{theorem}
  The dynamics (\ref{ep1})-(\ref{ep3}) either,
  \begin{enumerate}
  \item  converges to $(0,0,0)$ for generic initial conditions, or,
  \item does not equilibrate for generic parameter values. 
  \end{enumerate}
\end{theorem}

\begin{proof}

  If $(s, i, r)$ is an equilibrium of the above system, we perforce
  have $r = \frac{\beta}{\nu_2}i$ by setting the right hand side of
  (\ref{ep3}) to zero and therefore the equilibrium equations obtained
  by setting the right hand side of (\ref{ep1})-(\ref{ep2}) equal to
  zero reduce to
  \begin{displaymath}
    \lambda si = \Bigg(s + \left(1 +
    \frac{\beta}{\nu_2}\right)i\Bigg)\Bigg(\kappa s + \left(\upsilon_1
    +\frac{\upsilon_2\beta}{\nu_2}\right)i\Bigg) = \left(s +\Bigg(1 +
    \frac{\beta}{\nu_2}\right)i\Bigg)\Bigg((\gamma + \nu_1)i\Bigg).
  \end{displaymath}
  Then for any non-zero equilibrium, $i > 0$ and $s = \zeta i/\kappa$
  where $\zeta := \gamma + \nu_1 - \upsilon_1 -
  \frac{\upsilon_2\beta}{\nu_2}$.  This fixes the value of $\lambda$
  in terms of the other parameters. Thus for all but one value of
  $\lambda$, the dynamics does not equilibrate.
\end{proof}


\section{Optimal Control of the SIR-NC Epidemic}
\label{sec:SIR-NC-Control}

\subsection{The control problem}
\label{sec:control-intro}
We now discuss another key objective of this paper---to define
associated control problems. The following objectives are of obvious
interest to the controller.
\begin{itemize}
\item Increasing $T_{\max}$: this allows the healthcare system to
  prepare better to treat the infected population, especially those
  that may develop severe symptoms.
\item Decreasing $I_{\max}$: This reduces the preparation cost.
\item Decreasing $T_e$, the time to end the epidemic: This would mean
  that the system can get back to normalcy quicker.
\item Achieve herd immunity by $T_0.$
\end{itemize}
In addition one could also define more detailed objectives, e.g.,
define objectives similar to the preceding for specific communities
that may or may not interact. 

Two parameters are available for control; $\lambda :=$ the spreading
rate and $\gamma :=$ the removal rate. Recall that in the original SIR
model, $\lambda$ is the product of $\mu :=$ the rate at which people
meet (or more accurately, are in physical proximity), and $p :=$ the
probability of the proximity resulting in an infection. In SIR-NC, we
have separated the relative time scales of the two processes, but the
interpretation remains. While $p$ is more likely to be a feature of
the infection and the community, $\mu$ can be controlled through
measures such as imposing social distancing and varying degrees of
lockdown. Hence we let $\lambda$ take finitely many discrete
levels. Further, as we have argued earlier, the recovered population
is viewed as removal. In practice, increasing the removal rate is
effected by identifying the infected early, i.e., by increased testing
and quarantining of those that test positive. We assume that the
testing capacity can be continuously varied and thus allow $\gamma
\in$ a suitable interval of $\Re^{+}.$

Specifically, we assume that $\lambda \in L := \{\ell_1, \cdots ,
\ell_M\}$, i.e., takes discrete values, with switching cost $(\ell',
\ell'') \mapsto c(\ell', \ell'') > 0$ when $\ell' \neq \ell''$ and $=
0$ when $\ell' = \ell''$. It is customary and sensible to assume that
$c$ satisfies the `triangle inequality'
\begin{displaymath}
  c(\ell_i, \ell_k) < c(\ell_i, \ell_j) + c(\ell_j, \ell_k)
\end{displaymath}
for distinct $\ell_i, \ell_j, \ell_k \in L$, in order to avoid some
pathologies. On the other hand, $\gamma$ can be continuously
manipulated as a function of the current state. We first consider as
the state process $y(t) = \frac{I(t)}{N(t)} \in [0, 1]$ introduced
earlier. Later on we shall also consider the full state $(S(t),
I(t))$. We let $\gamma: [0, 1] \mapsto [0,\Gamma]$ for some $\Gamma
\geq 0$.

We consider a finite horizon control problem. There is a very good
motivation for doing so. For scenarios such as this, Model Predictive
Control (MPC) \cite{Mayne14} is often the paradigm of choice, where at
time $t$, one solves a finite horizon optimal control problem on time
horizon $[t, t + T]$ and uses the optimal control choice for time $t$
at time $t$. However, as time moves forward, to say $t' > t$, one
again solves a new finite horizon problem on $[t', t' + T]$ and the
process is repeated for each time instant. The basis of this procedure
is the solution of a finite horizon problem, so we have preferred it
over other alternatives.

\subsection{Continuous time control}
\label{sec:cont-control}
We begin with the continuous time formulation. Let $T > 0$ be the time
horizon under consideration. There is a per unit time running cost
$k(y(t), \lambda(t), \gamma(t)) \geq 0$ and a terminal cost $h(y(T))
\geq 0$. The objective is to minimize the cost
\begin{equation}
  \int_s^Tk(y(t), \lambda(t), \gamma(t))dt +
  \sum_{s \leq \tau_i \leq T}c(\lambda(\tau_i^-), \lambda(\tau_i)) +
  h(y(T)) \label{totalcost}
\end{equation}
for $s = 0$. Here $\{\tau_i\}$ are the times when $\lambda(t)$ makes a
discrete switch between two values in $L$, $\lambda(\tau_i^-),
\lambda(\tau_i)$ denoting resp., the value of $\lambda(\cdot)$ just
before the switch and just after the switch.  The value function then
will be a map
\begin{displaymath}
  V : (y', \ell, s) \in \mathcal{R}^+\times L \times \mathcal{R}^+
  \mapsto V(y', \ell, s) \in \mathcal{R}^+.
\end{displaymath}
defined as the infimum of (\ref{totalcost}) over all admissible
choices of $\lambda(t), \gamma(t), \tau_i , T \geq \tau_i; \tau_i,
t\geq s,$ with $y(s) = y', \lambda(s) = \ell_i$ .  For ease of
notation, we shall write $V(y, \ell_i, t)$ simply as $V^i(y,t)$ and
the corresponding running cost $k(y,\ell_i, \gamma)$ as $k^i(y,
\gamma)$. Since this is a mixed switching and continuous control
problem, the Hamilton-Jacobi system here is a system of
quasi-variational inequalities given by: for all $i$,
\begin{eqnarray}
  0 &\leq& \frac{\partial V^i}{\partial t}(y,t) + \min_{0 \leq u \leq
    \Gamma}\left((\ell_iy(1-y) - uy)\frac{\partial V^i}{\partial y}(y,t) +
  k^i(y,u)\right), \label{DPcont} \\
  0 &\leq& \min_{j \neq i}\left[c(\ell_i, \ell_j) + V^j(y,t)\right] -
  V^i(y,t),  \label{DPdiscr} \\
  0 &=& \left(\frac{\partial V^i}{\partial t}(y,t) + \min_{0 \leq u \leq
    \Gamma}\left((\ell_iy(1-y) - uy)\frac{\partial V^i}{\partial
    y}(y,t)\right) + k^i(y,u)\right)\times \nonumber\\
  && \left(\min_{j \neq i}\left[c(\ell_i, \ell_j) + V^j(y,t)\right] -
  V^i(y,t)\right), \label{product} \\
 &&  \ V^i(y,T) = h(y). \label{terminal}
\end{eqnarray}

The explanation is as follows. In absence of any switching of the
$\lambda$ parameter, (\ref{DPcont}) with equality would be the
standard Hamilton-Jacobi equation of dynamic programming for
continuous control. The same will be true now when you don't
switch. When it is better to switch, continuing with the continuous
control without switching will not be optimal, so inequality should
hold. This explains (\ref{DPcont}). Inequality (\ref{DPdiscr}) is in
similar spirit for the discrete control decision of switching:
equality holds when it is optimal to switch, inequality holds
otherwise. At any instant, it is either optimal to switch or to
continue with the optimal continuous control. This leads to
(\ref{product}). Finally, (\ref{terminal}) is the boundary condition
given by the terminal cost.

Given a solution of this system, one uses $\lambda(t) = \ell_i$ when
\begin{displaymath}
  V^i(y(t), t) < \min_{j \neq i}(c(\ell_i, \ell_j) + V^j(y(t), t),
\end{displaymath}
along with the continuous control $\gamma(t) :=$ a measurable
selection from the minimizers of
\begin{displaymath}
  u \mapsto \left((\ell_iy(t)(1-y(t)) - uy(t))\frac{\partial
    V^i}{\partial y}(y(t),t) + k^i(y(t),u)\right).
\end{displaymath}
Since we are in a finite horizon, a discretization can be solved by
simple backward recursion. 

Next we state an alternative
framework based on the actual state variables, i.e., the pair $(S(t), I(t))$. The
advantage of using $y(t)$ as state variable is that it is a bounded
scalar and therefore easier to handle. It also reflects the costs
associated with $I(t)$ somewhat faithfully. Nevertheless, it is not
the real thing and computational power permitting, it may be better to
stay with the actual variables. We shall use the same notation for the
cost functions and the value function as before. The quasi-variational
inequality then becomes
\begin{align}
  0 &\leq \frac{\partial V^i}{\partial t}(S,I,t) + \min_{0 \leq u
    \leq \Gamma}\left( \left(\frac{\ell_iSI}{S + I} -
  uI\right)\frac{\partial V^i}{\partial I}(S,I,t) + k^i(S,I,u)\right)
  \nonumber \\  
  & - \ \left(\frac{\ell_iSI}{S + I}\right)\frac{\partial
    V^i}{\partial S}, \label{DPcont2} \\  
  0 &\leq \ \min_{j \neq i}\left[c(\ell_i, \ell_j) + V^j(S,I,t)\right]
  - V^i(S,I,t), \label{DPdiscr2} \\  
  0 &=\ \Bigg(\frac{\partial V^i}{\partial t}(S,I,t) + \min_{0 \leq u
    \leq \Gamma}\left(\left(\frac{\ell_iSI}{S + I} -
  uI\right)\frac{\partial V^i}{\partial I}(S,I,t) + k^i(S,I,u)\right)
  \nonumber \\  
    & \ \ - \ \left(\frac{\ell_iSI}{S + I}\right)\frac{\partial
    V^i}{\partial S}\Bigg)\left(\min_{j \neq i}\left[c(\ell_i, \ell_j)
    + V^j(S,I,t)\right] - V^i(S,I,t)\right), \label{product2} \\ 
  & \ V^i(S,I,T) = h(S,I). \label{terminal2}
\end{align}

Comments
analogous to those following (\ref{DPcont})--(\ref{terminal}) also
apply here.

\subsection{Discrete time control problem}
\label{sec:discrete-control}
Unfortunately, the quasi-variational inequalities for first order
systems are difficult to analyze and are usually not well-posed in the
classical sense. For this reason, most work in this direction has been
in the framework of viscosity solutions
\cite{Dolcetta84,Dharmatti05,Zhang06}. We shall take recourse to a
simpler alternative, viz., to consider a discrete skeleton of the
problem, which leads to a simple backward recursion.

\NEW{We should mention here that there is already some work on pure
  impulse control of this model, i.e., without the continuous control
  \cite{Clancy99}, \cite{Piunovskiy19}. This is technically a little
  easier and does not run into some of the aforementioned technical
  issues.}

Let $a > 0$ be a small time-discretization step. A discretization of
the underlying dynamics is given by the Euler scheme for the original
continuous dynamics as:
\begin{eqnarray}
  S(t+1) &=& S(t) - a\left(\frac{\lambda(t)S(t)I(t)}{S(t) +
    I(t)}\right), \label{Sdiscr} \\
  I(t+1) &=& I(t) + aI(t)\left(\frac{\lambda(t)S(t)}{S(t) + I(t)} -
  \gamma(t)\right), \label{Idiscr}
\end{eqnarray}
for $t = 0, 1, 2, \cdots$. We again use the same notation as before
for cost functions and the value function. The objective is to
minimize
\begin{equation}
  \sum_{t=s}^{T-1}k(S(t), I(t), \lambda(t), \gamma(t)) +
  \sum_{s \leq \tau_i < T}c(\lambda(\tau_i^-), \lambda(\tau_i)) +
  h(S(T),I(T)) \label{totalcost2}  
\end{equation}
for $s = 0$. The value function $V^i(S', I', s)$ is then the infimum
of (\ref{totalcost2}) over all admissible $\lambda(t), \gamma(t),
\tau_i, T \geq t, \tau_i \geq s$, beginning at $\lambda(s) = \ell_i$
and $S(s) = S', I(s) = I'$. The quasi-variational inequalities can be
written along the lines of the foregoing, but for discrete systems,
they can be put in the form of a very convenient backward
recursion. Let $Z(t) = (S(t), I(t))$ and write the above discrete
dynamics (\ref{Sdiscr})-(\ref{Idiscr}) compactly as
\begin{equation} 
  Z(t+1) = F(Z(t), \lambda(t), \gamma(t)), \ t \geq 0. \label{dynamics}
\end{equation}
Then the discrete quasi-variational inequalities reduce to the simple
backward recursion of dynamic programming given by
\begin{equation}
  V^i(z,t) = \min_{j,u}\Big( c(\ell_i, \ell_j) + k^j(z, u, t) +
  V^j(F(z, \ell_j, u), t+1)\Big), \ 0 \leq t < T, \label{DPfinal}
\end{equation}
with terminal condition $V^i(z,T) = h(z)$.

\noindent \textbf{Remark} \eqref{totalcost2} can also be used for the
population conserving SIR model with an appropriate version of the
discrete dynamics for \eqref{dynamics}. We do not pursue this here.

If the objective is to minimize the peak over this interval, we
augment the dynamics (\ref{Sdiscr})-(\ref{Idiscr}) with an additional
state variable, viz., the running maximum $M(t) := \max_{0 \leq s \leq
  t}I(s)$ with its update equation
\begin{equation}
  M(t+1) = \max(M(t), I(t+1)). \label{maxeq}
\end{equation}
The new state now is $Z(t) = (S(t), I(t), M(t))$ with dynamics
\begin{displaymath}
  Z(t+1) = F(Z(t), \lambda(t), \gamma(t))
\end{displaymath}
for a suitably redefined $F$. The dynamic programming equation is the
same as before with $k^i \equiv 0 \ \forall i$ and $h((s,i,m)) = m$.

\subsubsection{Numerical Examples}
\label{sec:Control-example}

We now illustrate the preceding control problem via numerical
examples. In our computations, we minimize \eqref{totalcost2} subject
to \eqref{dynamics} with a time-discretization step of $a=1,$ treating
it as an optimization problem, and use a simple forward tree search to
solve it.

We consider a community with $N=10,000,$ natural $\lambda$
of $1/4,$ and natural $\gamma,$ designated as $\gamma_0,$ of 
$1/15.$ Recall from the example in Section~\ref{sec:SIR-NC-model}
that for this system $T_{\max}=55.4$ and $I_{\max}=4523.$ We assume
that the control objective at time $t$ would be to have the number of
infected at time $(t+T)$ to be a fraction $\alpha(t+T)$ of the value
in the uncontrolled system, which we will denote by $I_o(t+T).$ With
this view, we use the following cost function $h(\cdot)$
\begin{displaymath}
  h(S(t+T),I(t+T)) =  A_1 e^{a_1 (I(t+T)-\alpha I_o(t+T))}, 
\end{displaymath}
with $a_1$ $A_2$ to be constants. This function penalizes missing the
the target significantly more than the reward for exceeding it. In our
computations we have used $a_1=10,$ $A_1=100,$ and $T=3.$ To model an
aggressive control objective we use $\alpha(t)=0.1$ for all $t.$ To
model a potentially more modest objective, we use
\begin{displaymath}
  \alpha(t) = \begin{cases}
    1-\frac{0.9}{T_{\max}} t & \mbox{for $0 \leq t \leq 55$}\\
    0.1 & \mbox{for $t> T_{\max}$}
  \end{cases}
\end{displaymath}
This corresponds to slowly increasing the requirement reaching
$0.1I_{\max}$ at $T_{\max}.$ 

We assume $M=3,$ i.e., there are three contact rates among the members
of the community corresponding to three levels of lockdown. We 
use $L=\{1/4, 3/16, 1/8\}$ corresponding to, respectively, no
lockdown, a partial lockdown, and a total lockdown. The switching
costs are
\begin{displaymath}
  c = \begin{pmatrix}
    0    & 2a_2   & 3a_2\\
    0.1a_2 & 0    & 2a_2\\
    0.1a_2 & 0.1a_2 & 0
  \end{pmatrix}
  .
\end{displaymath}
The cost of switching to a lower $\lambda$ is significantly higher
than to switch to a higher $\lambda.$ We use $a_2=10,000$ in the
simulations. 

The running cost, $k(\cdot),$ will have the following additive
components.
\begin{itemize}
\item Cost due to the chosen level of lockdown; we will use $L_1=0,$
  $L_3=10L_2$ corresponding to the costs of the three levels of the
  lockdown. We use $L_2=10,000.$

\item Cost of testing and quarantining. This will be an increasing
  function of the rate of removal of the infected subpopulation in
  excess of the natural rate of $1/15$ and the number that are
  infected, $I(t).$ At the beginning of the epidemic, it may be hard
  to obtain the testing kits and to also define the protocol and will
  hence be very costly to achieve a high rate of removal. Eventually,
  the time dependent cost becomes negligible. This leads us to propose
  the following function for this component: $(A_2 + \frac{A_3}{t})
  I(t) \left( \gamma(t) - \gamma_0 \right).$ We use $A_2=A_3=100$ in
  the examples. Further, we discretise the feasible $\gamma(t)$ to 10
  levels in the interval $[1/15, 1/3]$

\item A fraction of the infected will be hospitalised. There may be
  other knock-on effects of prolonged illnesses due to the infection,
  less attention to other endemic diseases like tuberculosis,
  hypertension, and diabetes. Finally, if $I(t)$ is allowed to become
  too large, hospital space may become unavailabele. This suggests
  $A_4I(t)^{a_3}$ to be the form for this component. In our
  computations, we have used $A_4=10$ and consider $a_3$ in the range
  $(1,1.5).$ There is an interesting change in the behaviour of the
  cost functions and the control variable as $a_3$ is increased from
  1.0. We will see this later.
 
\end{itemize}

From the preceding discussion, we have
\begin{displaymath}
  k(S(t), I(t), \lambda(t), \gamma(t)) = L_{\lambda(t)} + \left(A_2 +
  \frac{A_3}{t} \right) I(t) \left( \gamma(t) - \gamma_0 \right) + A_4
  I(t)^{a_3}
\end{displaymath}

\begin{figure}[tbp]
  \begin{subfigure}{0.32\textwidth}
    \centering
    \includegraphics[width=\textwidth]{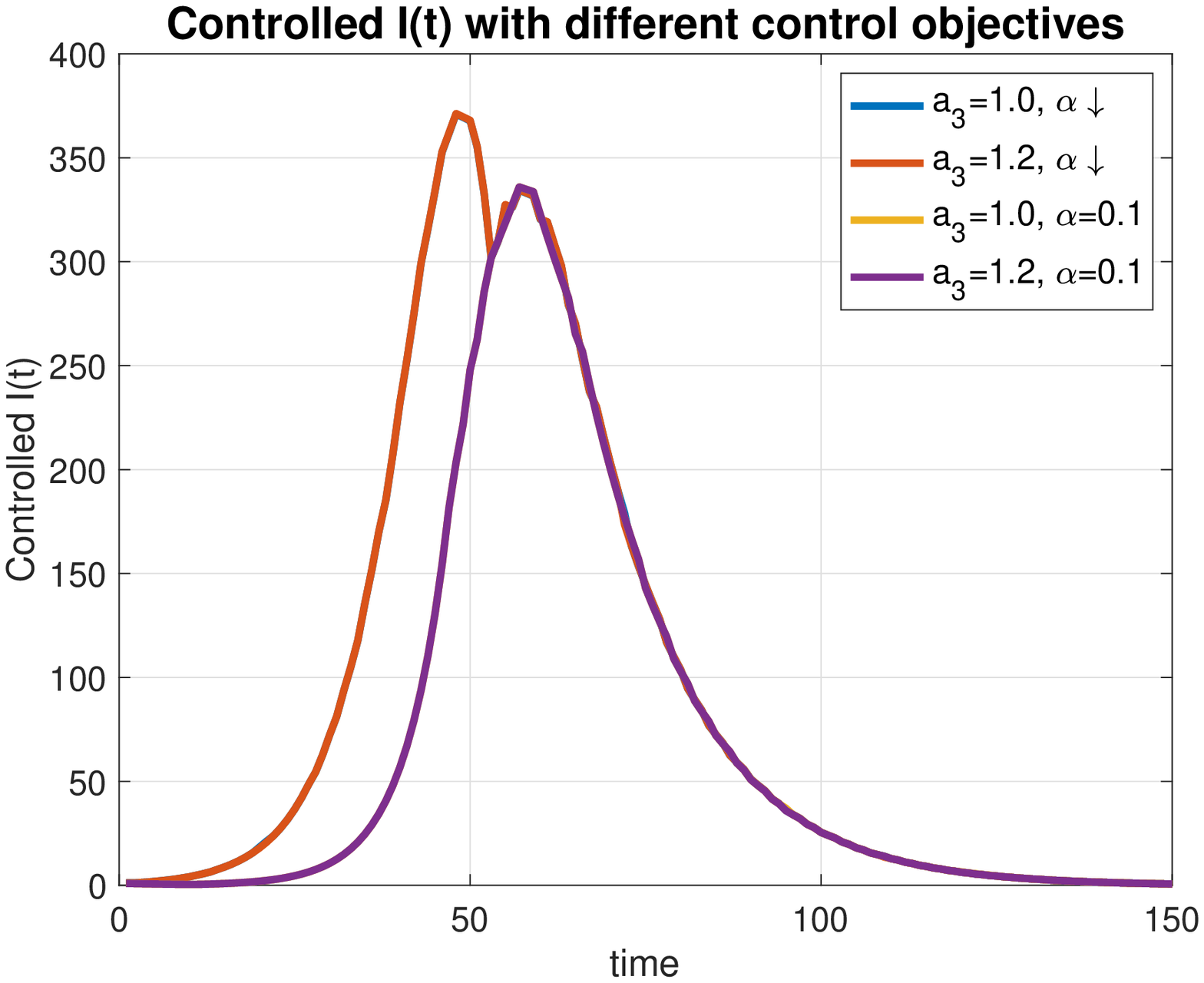}
    \caption{Controlled $I(t)$.}
  \end{subfigure}
  \begin{subfigure}{0.32\textwidth}
    \centering
    \includegraphics[width=\textwidth]{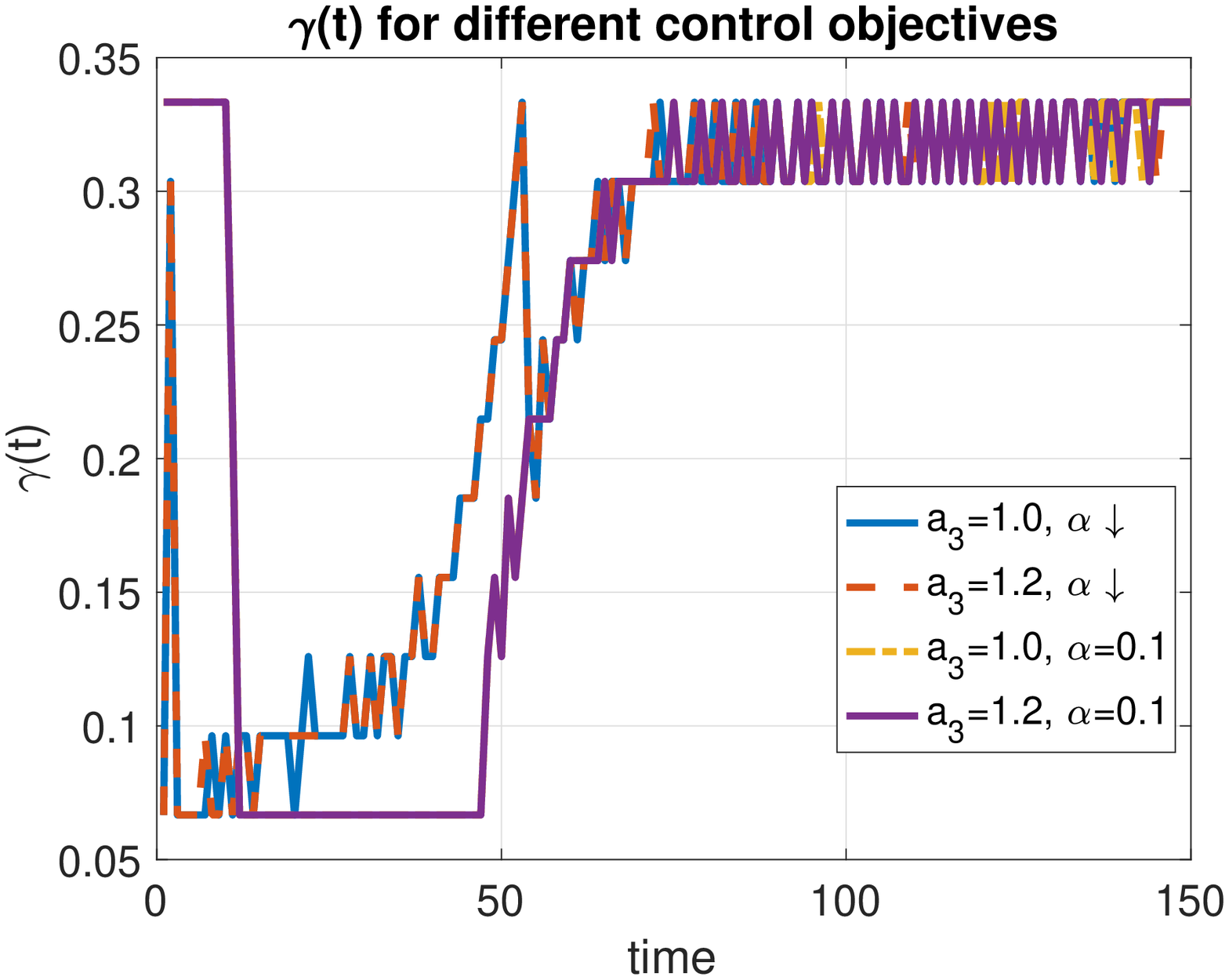}
    \caption{Controlled $\gamma(t)$.}
  \end{subfigure}
  \begin{subfigure}{0.32\textwidth}
    \centering
    \includegraphics[width=\textwidth]{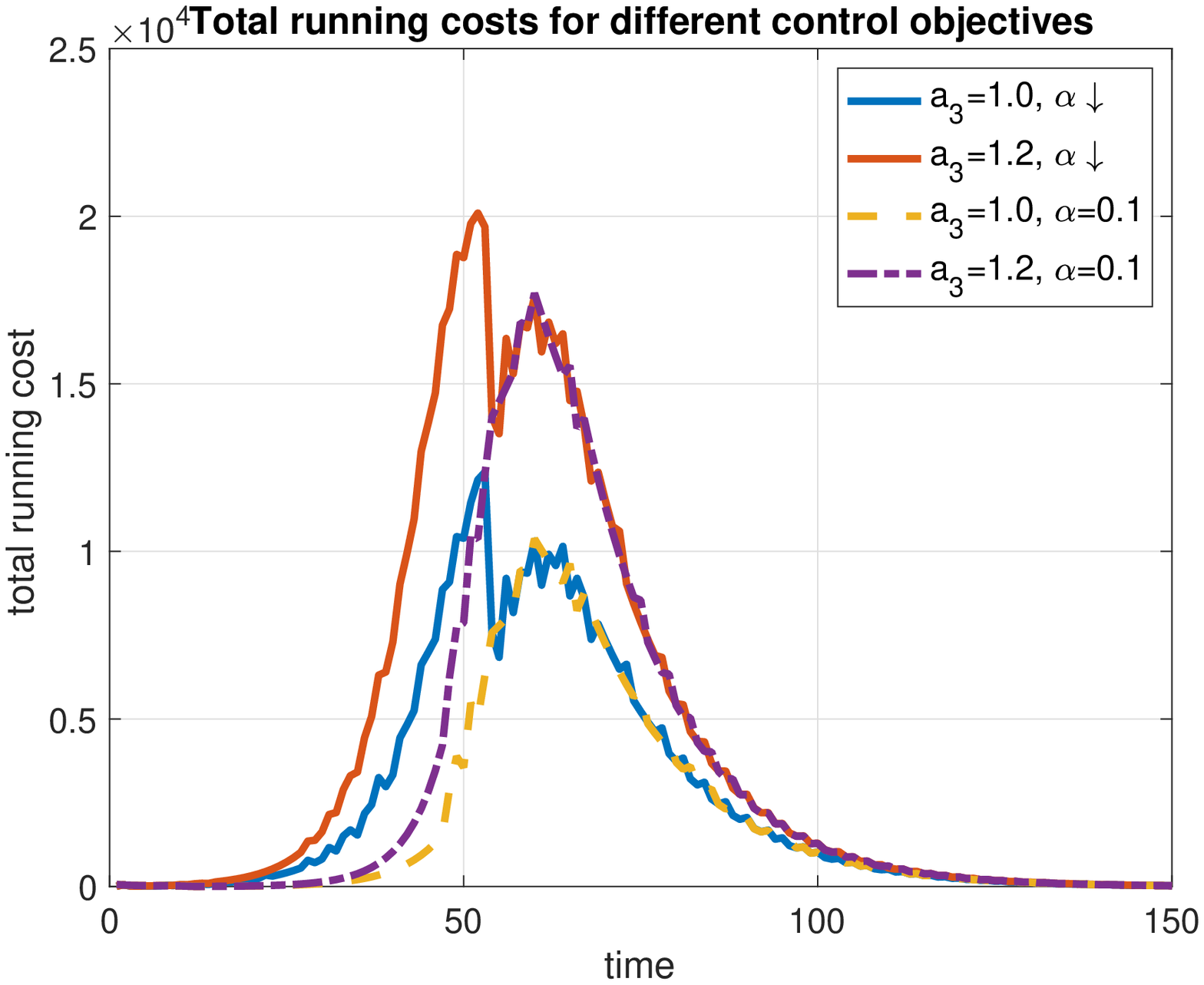}
    \caption{Sum of Switching and running cost in each time slot..}
  \end{subfigure}

  \caption{Controlled $I(t),$ $\gamma(t),$ and total running cost for
    $a_3 =1.0, 1.2,$ and $\alpha(t)$ decreasing form $1.0$ to $0.1$ in
    the period $(0, T_{\max})$ and $\alpha(t)=0.1.$}
  \label{plot:control1}
\end{figure}

\begin{figure}[tbp]
  \begin{subfigure}{0.32\textwidth}
    \centering
    \includegraphics[width=\textwidth]{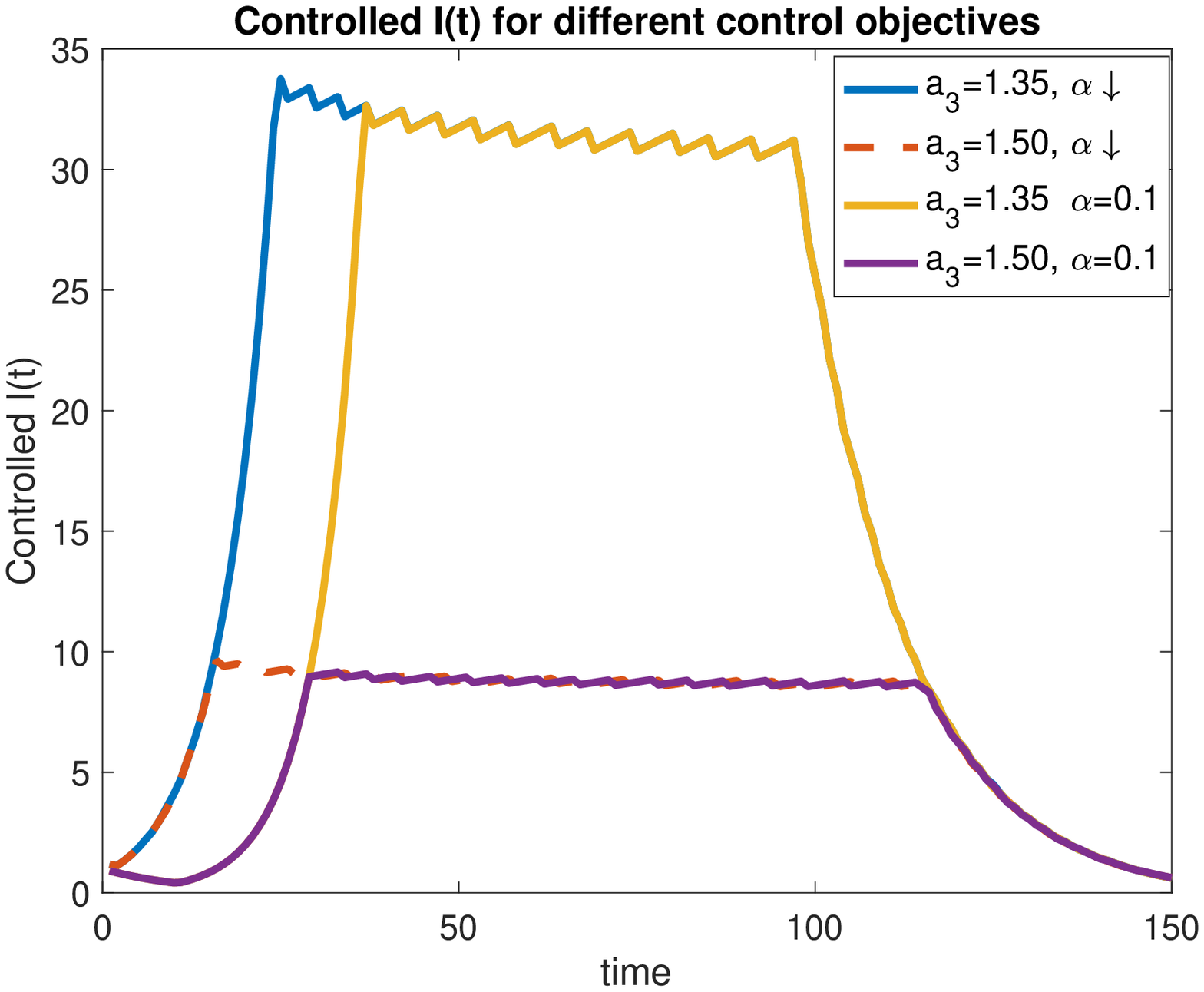}
    \caption{Controlled $I(t)$.}
  \end{subfigure}
  \begin{subfigure}{0.32\textwidth}
    \centering
    \includegraphics[width=\textwidth]{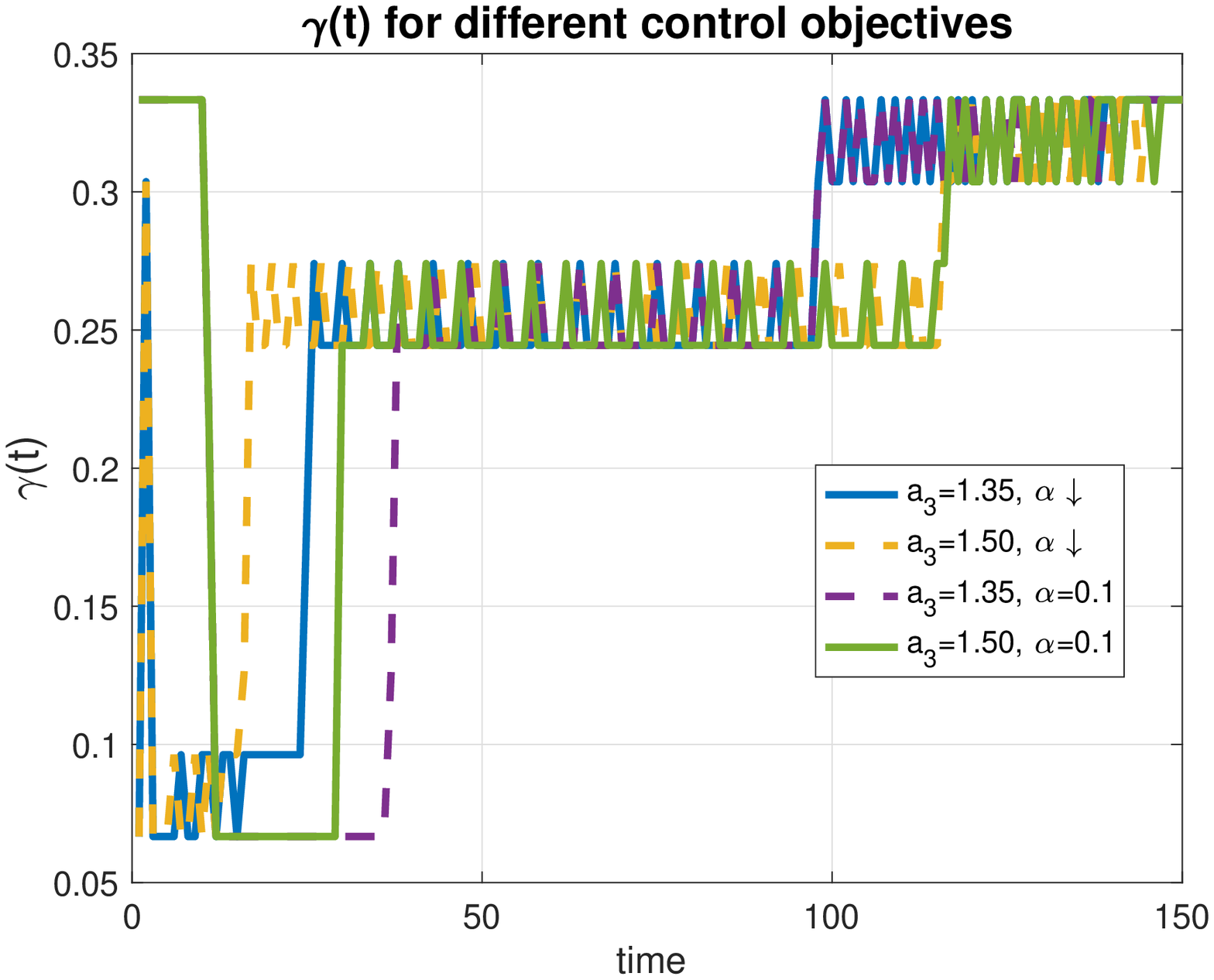}
    \caption{Controlled $\gamma(t)$.}
  \end{subfigure}
  \begin{subfigure}{0.32\textwidth}
    \centering
    \includegraphics[width=\textwidth]{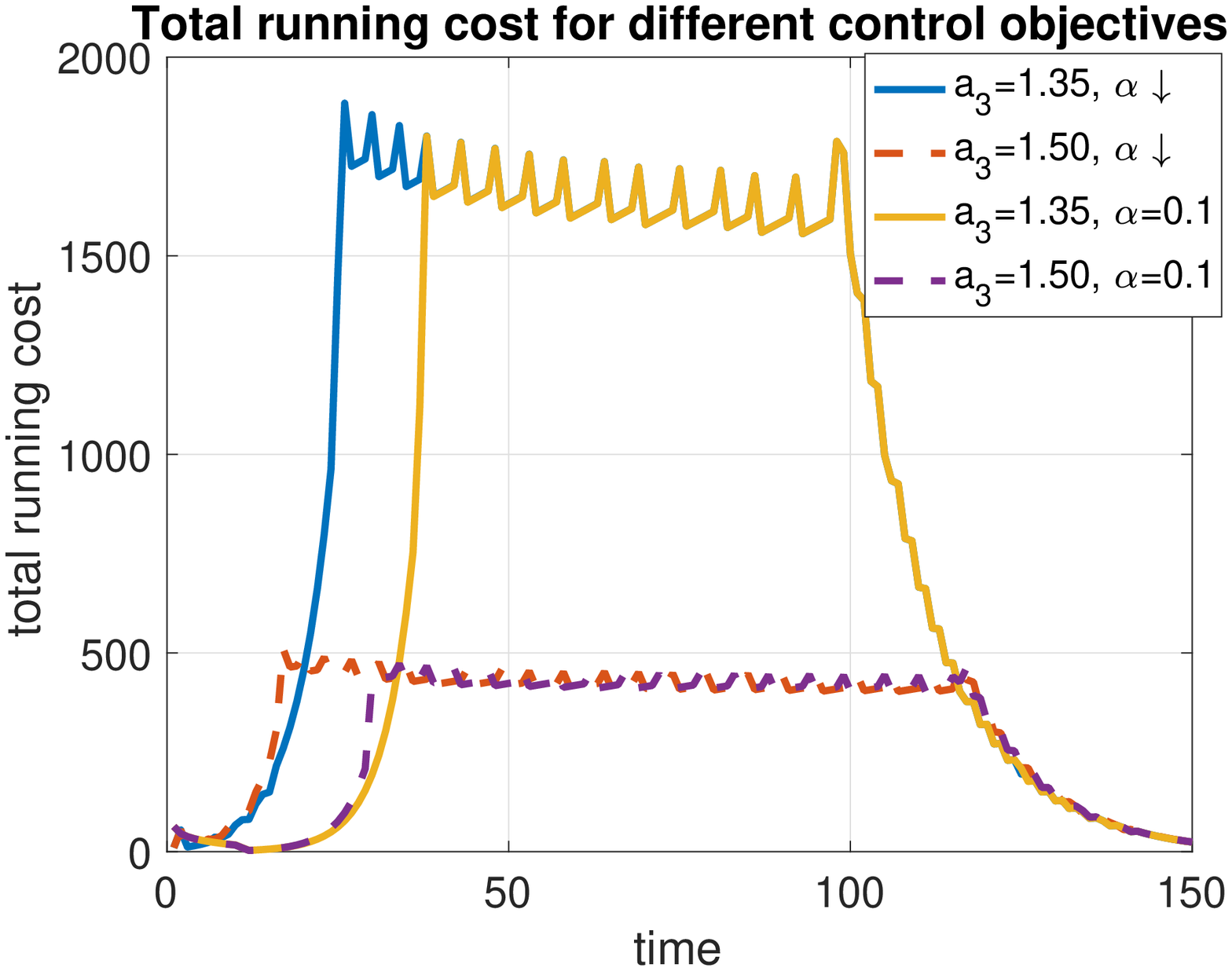}
    \caption{Sum of Switching and running cost in each time slot..}
  \end{subfigure}

  \caption{Controlled $I(t),$ $\gamma(t),$ and total running cost for
    $a_3 =1.35, 1.5,$ and $\alpha(t)$ decreasing form $1.0$ to $0.1$
    in the period $(0, T_{\max})$ and $\alpha(t)=0.1.$}
  \label{plot:control2}
\end{figure}

In Figs.~\ref{plot:control1} and ~\ref{plot:control2} we plot $I(t),$
$\gamma(t)$ and the running cost for different $a_3$ and $\alpha(t).$
There are several takeaways from these results. For the parameters
chosen, the switching costs and the running cost of reduced $\lambda$
was too high for the system to switch $\lambda,$ thus $\lambda(t)=1/4$
was the control for all the parameters. As we changed $a_3$ from $1.0$
to 1.2, for decreasing $\alpha(t),$ the change in both $I(t)$ and
$\gamma(t)$ was barely noticeable. This was the case even when
$\alpha(t)$ was held constant at $0.1.$ However as we changed $a_3$
from $1.2$ to $1.35$ the changes were significant both in shape and
magnitude of $I(t),$ $\gamma(t),$ and the total running cost. This can
be seen when comparing the plots in Figs.~\ref{plot:control1} and
~\ref{plot:control2}. A higher $a_3$ induces higher levels of control
and hence reduced $I(t)$ this in turn reduced the running
cost. Observe that there is an order of magnitude reduction in $I(t)$
and in the cost. Thus a higher `valuation' for infections induces a
higher level control which in turn reduces costs in the long run. 

\subsection{Multi-objective and multi-community control}
\label{sec:multi-objective}
As we have seen at the beginning of this section, there are several
competing criteria we may want to optimize and focusing on just one at
the expense of others may not be wise. We therefore propose here a
multi-objective control scheme inspired by \cite{Shah18}. We shall
assume that the state space $\Re^+ \times\Re^+$ is now suitably
discretized and denoted by $S$. Likewise, the dynamics
(\ref{dynamics}) is suitably replaced by a dynamics on $S$, where we
retain the same notation as (\ref{dynamics}) for sake of simplicity.
For example, the original $F$ could be replaced by $F$ followed by a
suitable vector quantization scheme. As already noted, some cost
criteria may require a suitable augmentation of the state space.

Consider $K > 1$ running cost functions $k^i_1(\cdot, \cdot, \cdot),
\cdots , k^i_K(\cdot, \cdot, \cdot))$, switching cost functions
$c_1(\cdot, \cdot), \cdots , c_K(\cdot, \cdot)$ and terminal cost
functions $h_1(\cdot), \cdots , h_K(\cdot)$. Let $C_1,$ $\cdots,$ $C_K
> 0$ denote prescribed numbers. Let $(i,z,t) \mapsto V_j^i(z,t), 1
\leq j \leq K,$ denote the value function corresponding to the $j$th
running cost function $k_j^\cdot(\cdot, \cdot, \cdot)$, switching cost
function $c_j(\cdot, \cdot)$, and terminal cost $h_j(\cdot)$.  Our
objective will be to ensure that for given initial condition $z = z_0$
of the state and the initial choice $\ell_{i_0}$ of the discrete
control variable,
\begin{equation}
  V_j^{i_0}(z_0,0) \leq C_j \ \forall j. \label{constraint}
\end{equation}
We assume that the set of admissible controls for which
(\ref{constraint}) holds is not empty. We scalarize this problem by
considering a single running cost function $(\ka_w^\cdot(\cdot, \cdot,
\cdot) := \sum_jw_jk^i_j(z,u,t)$, a single switching cost
$\ca_w(\cdot, \cdot) := \sum_jw_jc_j(\cdot, \cdot)$ and a single
terminal cost function $\ha_w(\cdot) := \sum_jw_jh_j(\cdot)$, where
$w_0 := [w(1), \cdots , w(K)]^T$ be a vector satisfying $w(j) > 0
\ \forall j$ and $\sum_jw(j) = 1$ (i.e., a probability vector).  The
algorithm then is as follows. For $n \geq 0$, do the following.

\begin{enumerate}

\item Perform the $K$ backward recursions
  \begin{equation}
    V^i_m(z,t) = \min_{j,u}\Big( c_m(\ell_i, \ell_j) + k_m^j(z, u,
    t) + V^j_m(F(z, \ell_j, u), t+1)\Big), \ 0 \leq t <
    T, \label{DPseparate}
  \end{equation}
  with terminal condition $V_{m,n}^i(z,T) = h_m(z)$, $1 \leq m \leq
  K$.

  For $n \geq 0$, do the following.
  
\item  Perform the additional backward recursion
  \begin{align}
    \V_n^i(z,t) & = \min_{j,u}\Big( \ca_{w_n}(\ell_i, \ell_j) +
    \ka_{w_n}^j(z, u, t) + \V_n^j(F(z, \ell_j, u), t+1)\Big), \\
    & \hspace{0.5in} \ 0 \leq t  < T, \label{DPcombined}
  \end{align}
  with terminal condition $\V_n^i(z,T) = \ha_{w_n}(z)$.

\item Perform a single update of $w_n = [w_n(1), \cdots , w_n(K)]^T$
  as
  \begin{equation}
    w_{n+1}(m) = w_n(j)+ a(n)w_n(m)\Big((V^{i_0}_m(z_0,0) - C_j) -
    (\V^{i_0}_n(z_0,0) - \C(w_n)\Big) \label{repl}
  \end{equation}
  for $1 \leq m \leq K$, where $\C(w) := \sum_mw(m)C_m$ and $a(n) > 0$
  satisfy
  \begin{equation}
    \sum_ma(m) = \infty. \label{stepsize}
  \end{equation}
\end{enumerate}

Denote by $V^i_w$ the unique solution to
\begin{align*}
  \V_w^i(z,t) & = \min_{j,u}\Big( \ca_{w}(\ell_i, \ell_j) +
  \ka_{w}^j(z, u, t) + \V_w^j(F(z, \ell_j, u), t+1)\Big), \\
  & \hspace{0.5in} \ 0 \leq t  < T.
\end{align*}
Iteration (\ref{repl}) is simply the Euler scheme for the differential
equation
\begin{equation}
  \dot{w}_m(t) = w_m(t)\Big((V^{i_0}_m(z_0,0) - C_j) -
  (\V^{i_0}_{w(t)}(z_0,0) - \C(w(t))\Big), \ 1 \leq m \leq K,
  \label{replODE}
\end{equation}
where $V^i(z, t) := \sum_mw_m(t)V^i_m(z, t)$. Let $(q_n(\cdot,
\cdot),v_n(\cdot, \cdot))) : (\Re^+)^2\times \{0, 1, \cdots, T\}
\mapsto L\times[0, \Gamma], n \geq 0,$ be a sequence of stationary
policies such that $(q_n(z,m), v_n(z,m))$ attains the minimum on the
right hand side of (\ref{DPcombined}).

\begin{theorem}
  Every stationary policy obtained as a limit point of $(q_n,v_n)$ as
  $n\uparrow\infty$ satisfies (\ref{constraint}).
\end{theorem}

This is proved exactly as in \cite{Shah18}. In fact, the scenario here
is vastly simpler than that of \textit{ibid.} due to absence of noise,
asynchrony, etc. We only sketch the key steps here, referring the
reader to \cite{Shah18} for further details. First, observe that
(\ref{replODE}) is a special case the replicator dynamics of
evolutionary biology \cite{Sandholm10}.  As in \cite{Shah18}, the
discrete iteration (which can be viewed as the Euler scheme for
(\ref{replODE}) with slowly decreasing stepsizes) tracks the
asymptotic behavior of (\ref{replODE}). The condition (\ref{stepsize})
plays a key role here: since $a(n)$ plays the role of a time step in
the time discretization of (\ref{replODE}), (\ref{stepsize}) ensures
that the entire time axis is covered, so that the asymptotic behavior
of (\ref{replODE}) can be tracked. We do not, however, need the
standard condition in stochastic approximation algorithms that
$\sum_na(n)^2 < \infty$ which features in \cite{Shah18}, because we
have a purely deterministic iteration without any noise. Thus the
problem is reduced to establishing that (\ref{replODE}) indeed
converges to the desired set. This is already proved in \cite{Shah18}.

Next we consider the control of two interacting communities like in
Section~\ref{sec:SIR-NC-communities} with the stipulation that the
dynamics of one evolves on a much faster scale than the other. One
possible scenario is a dense slum interacting with a geographically
sparse affluent section. Following
Section~\ref{sec:SIR-NC-communities}, we model this as the coupled
system
\begin{eqnarray}
  \frac{dS_a(t)}{dt} & = & - S_a(t) \left( \lambda_a(t)
  \frac{I_a(t)}{N_a(t)} + \lambda_{ab}(t) \frac{I_b(t)}{N_b(t)}
  \right) \label{an1} \\ 
 \frac{dI_a(t)}{dt} & = & S_a(t)\left( \lambda_a(t)
 \frac{I_a(t)}{N_a(t)} + \lambda_{ab}(t) \frac{I_b(t)}{N_b(t)} \right)
 - \gamma_a(t) I_a(t) \label{an2} \\ 
  \frac{dS_b(t)}{dt} & = & \varepsilon \Big(- S_b(t) \left(
  \lambda_b(t) \frac{I_b(t)}{N_b(t)} + \lambda_{ab}(t)
  \frac{I_a(t)}{N_a(t)} \right)\Big)
  \label{bn1}   \\
  \frac{dI_b(t)}{dt} & = & \varepsilon\Big( S_b(t) \left( \lambda_b(t)
  \frac{I_b(t)}{N_b(t)} + \lambda_{ab}(t) \frac{I_a(t)}{N_a(t)}
  \right) - \gamma_b(t) I_b(t)\Big) \label{bn2} \\  
  N_a & = & S_a(t) + I_a(t)  \label{an3} \\
  N_b & = & S_b(t) + I_b(t) \label{bn3}
\end{eqnarray}

Here $0 < \varepsilon << 1$ is a small parameter that renders the
evolution of $(S_b(t), I_b(t))$ to be on a slower time
scale. Following the standard philosophy for analyzing two time scale
systems, we view the slow time scale dynamics as quasi-static for
purposes of analyzing the dynamics on the fast time scale and treat
the fast time scale dynamics as quasi-equilibrated for the purposes of
analyzing the slow dynamics. Thus for the former, we consider the
system
\begin{eqnarray}
  \frac{dS_a(t)}{dt} & = & - S_a(t) \left( \lambda_a(t)
  \frac{I_a(t)}{N_a(t)} + \lambda_{ab}(t) Z
  \right), \label{af1} \\
  \frac{dI_a(t)}{dt} & = & S_a(t)\left( \lambda_a(t)
  \frac{I_a(t)}{N_a(t)} + \lambda_{ab}(t)Z \right) - \gamma_a(t)
  I_a(t), \label{af2}
 \end{eqnarray}
where we view $\frac{I_b(t)}{N_b(t)} \approx$ a constant $Z$. This
coincides with our model of SIR-NC with imported infections. If we
follow the standard procedure described above, then under the
condition $\lambda_{ab} < \lambda_a - \gamma_a$, the asymptotic
equilibrium is the origin, i.e., extinction, which means that the fast
variables simply drop out and the slow dynamics can be analyzed as a
single community dynamics. To get something more informative than
this, we propose the following, particularly in the context of the
control problem. Choose $\varepsilon < a << T$ above so that the time
horizon $T > 0$ is divisible by $a$ and $a$ is divisible by
$\varepsilon$. (If not, we can replace $T/a, a/\varepsilon$ by $\lceil
T/a\rceil, \lceil a/\varepsilon\rceil$ resp.\ in what follows.)

Consider the discrete time intervals $K_n := \{t(n), t(n) +
\varepsilon a, \cdots , t(n+1)\}$ for $0 \leq n < N_T :=
\frac{T}{a\varepsilon}$.  On the interval $K_n$, beginning with $n =
N_T - 1$ and going backward in time, do the following:
\begin{enumerate}
\item Consider the discretized dynamics for state variables $(S_a(n),
  I_a(n))$ corresponding to \eqref{af1}--\eqref{af2} for each possible
  value of $Z := I_b(t(n+1))/N_b(t(n+1))$.

\item Solve the dynamic programming equation on the discrete time
  interval $J_n := \{t(n), t(n) + a/\varepsilon, t(n) +
  2a/\varepsilon, \cdots , t(n+1)\}$ for the discretized dynamics of
  (\ref{af1})-(\ref{af2}) with terminal cost $:=$ the time $t(n+1)$
  value function inherited from the preceding computation on $[t(n+1),
    t(n+2)]$ if $t(n) < N_T$, and $:= h$ otherwise, so as to obtain
  the value function for $m \in J_n$ and initial (i.e., at time
  $T(n)$) condition $(s, i) \in (\Re^+)^2$.

Compute the optimal trajectory of $(S_a(m), I_a(m)), m \in J_n$.

\item Solve the dynamic programming equation for the discretized
  dynamics for $(S_b(m), I_b(m)), 0 \leq m \leq N_T,$ with $(S_a(m),
  I_a(m)) =$ the value thereof at $t(m)$ for the optimal trajectory in
  the preceding step on $[t(m), t(m+1)]$ ($=$ the initial condition if
  $m = 0$).

\end{enumerate}

This emulates the two time scale phenomenon above on two superimposed
discrete skeletons of the time axis, one being a refinement of the
other.

\section{Discussion and Conclusions}

We have proposed a variant of the classical SIR model wherein we drop
the key assumption of conservation of population size. This
Non-Conservative SIR or SIR-NC for short, has the advantage of being
analytically tractable, leading to explicit expressions in the simple
cases, and in addition, being more realistic for lethal pandemics. We
also introduced several variations of the basic model, such as
inclusion of natural birth and death rates, external inputs,
interacting communities, etc. While some preliminary results have been
presented for these, they open up many new directions for future work.

Finally, we consider controlled versions of the above. After
introducing a variety of control objectives and arguing that Model
Predictive Control is a natural framework for pandemic control, we
propose a finite horizon control problem. First we do so for the
original continuous time case, leading to a system of
quasi-variational inequalities due to co-occurrence of continuous
controls with discrete controls that also have switching costs. We
point out the technical difficulties therein and resort to a discrete
version, leading to drastic simplification. We also consider
multi-objective control in this framework. Finally, we briefly touch
upon interacting communities with significantly separated time scales
and propose an approximation scheme for analyzing the associated
control problems. Once again, this opens up rich possibilities for
future work.


\bibliographystyle{siamplain}
\bibliography{../epidemics}

\end{document}